\documentclass[11pt]{amsart}
\usepackage{amsmath,amssymb,amsfonts}
\usepackage{lmodern}
\usepackage{mathtools}
\usepackage{caption}
\usepackage{graphicx,subfigure}
\newcommand{\lebn}

\theoremstyle{plain}
\newtheorem{prop}[equation]{Proposition}

\newtheorem{thm}[equation]{Theorem}

\newtheorem{cor}[equation]{Corollary}
\newtheorem{lem}[equation]{Lemma}
\newtheorem{rem}[equation]{Remark}

\theoremstyle{definition}

\usepackage{chngcntr}
\newtheorem{nthm}{Theorem}
\newtheorem{nprop}[nthm]{Proposition}
\newtheorem{nlem}[nthm]{Lemma}
\counterwithin*{nthm}{subsection}

\newtheorem{claim}{Claim}
\newenvironment{cproof}[1]{\par\indent{\textit{Proof of the Claim.}}\space#1}{\hfill $\diamondsuit$}

\numberwithin{equation}{section}

\usepackage{tikz}
\usepackage{verbatim}
\usetikzlibrary{shapes,shadows,calc}
\usepgflibrary{arrows}
\usetikzlibrary{arrows, decorations.markings, calc, fadings, decorations.pathreplacing, patterns, decorations.pathmorphing, positioning}
\tikzset{nodc/.style={circle,draw=blue!50,fill=pink!80,inner sep=4.2pt}}
\tikzset{noddee/.style={circle,draw=black,fill=black,inner sep=1.6pt}}
\tikzset{nodel/.style={circle,draw=black,inner sep=2.2pt}}
\tikzset{nodinvisible/.style={circle,draw=white,inner sep=2pt}}
\tikzset{nodpale/.style={circle,draw=gray,fill=gray,inner sep=1.6pt}}

\tikzset{nodde/.style={circle,draw=blue!50,fill=pink!80,inner sep=4.2pt}}
\tikzset{noddee/.style={circle,draw=black,fill=black,inner sep=2pt}}
\tikzset{noddee1/.style={circle,draw=black,fill=black,inner sep=1.5pt}}
\tikzset{noddee2/.style={circle,draw=black,fill=black,inner sep=1pt}}

\tikzset{nod1/.style={circle,draw=black,fill=black,inner sep=1pt}}
\tikzset{nod2/.style={circle,draw=black,fill=blue!25!black,inner sep=1.6pt}}
\tikzset{nod3/.style={circle,draw=black,fill=black,inner sep=3pt}}
\tikzset{noddiam/.style={diamond,draw=black,inner sep=2pt}}
\tikzset{nodw/.style={circle,draw=black,inner sep=2pt}}
\usepackage{scalefnt}
\usetikzlibrary{arrows,decorations.pathmorphing,backgrounds,positioning,fit,petri}

\usepackage[colorlinks, urlcolor=blue, linkcolor=blue, citecolor=blue]{hyperref}

\newcommand{\N}{\mathbb{N}}

\newcommand{\sm}{\setminus}

\newcommand{\F}{\mathcal{F}}

\DeclareMathOperator*{\EFC}{EFC}
\DeclareMathOperator*{\ESE}{ESE}
\DeclareMathOperator*{\ECE}{ECE}

\DeclareMathOperator*{\VCE}{VCE}

\makeatletter
 \def\@textbottom{\vskip \z@ \@plus 10pt}
 \let\@texttop\relax
\makeatother

\usepackage[T1]{fontenc}
\usepackage[utf8]{inputenc}

\begin{document}

\title{Critical Equimatchable Graphs}

\author{Zakir Deniz}\thanks{Department of Mathematics, Duzce University, Duzce, Turkey, zakirdeniz@duzce.edu.tr}
\author{T{\i}naz Ekim}\thanks{Department of Industrial Engineering, Bo\u{g}azi\c{c}i University, Istanbul, Turkey, tinaz.ekim@boun.edu.tr }


\date{\today}

\thanks{The support of TUBITAK (grant no:121F018) is greatly acknowledged.}

\begin{abstract}
A graph $G$ is \emph{equimatchable} if every maximal matching of $G$ has the same cardinality.
In this paper, we investigate equimatchable graphs such that the removal of any edge harms the equimatchability, called \emph{edge-critical equimatchable} graphs ($\ECE$-graphs).  We show that apart from two simple cases, namely bipartite $\ECE$-graphs and even cliques, all $\ECE$-graphs are 2-connected factor-critical. Accordingly, we give a characterization of factor-critical $\ECE$-graphs with connectivity 2. Our result provides a partial answer to an open question posed by Levit and Mandrescu \cite{LM2016} on the characterization of well-covered graphs with no shedding vertex. We also introduce equimatchable graphs such that the removal of any vertex harms the equimatchability, called \emph{vertex-critical equimatchable} graphs ($\VCE$-graphs). To conclude, we enlighten the relationship between various subclasses of equimatchable graphs (including $\ECE$-graphs and $\VCE$-graphs) and discuss the properties of factor-critical $\ECE$-graphs with connectivity at least 3.
\end{abstract}
\maketitle


\section{Introduction}\label{sec:intro}

Matching theory is one of the fundamental fields that encompasses both practical and theoretical challenges \cite{LP}.  Given a graph $G$, a \emph{matching} is a set of edges of $G$ having pairwise no common endvertices. It is well-known that given a graph, a matching of maximum size can be efficiently computed whereas finding an inclusion-wise maximal matching of minimum cardinality is an NP-complete problem even in several restricted cases \cite{yannakakis}. 

A graph $G$ is called \emph{equimatchable} if every maximal matching of $G$ has the same cardinality. The structure of equimatchable graphs has been widely studied in the literature (see for instance \cite{ClawfreeEqm, triangle-free-equim, eqm_regular, ECEbip, oddcycles-eqm, EK2, Favaron, k-eqm, genus, Plummer}). The counterpart of equimatchable graphs for independent sets is called well-covered graphs: a graph is \emph{well-covered} if all its maximal independent sets have the same size. Well-covered graphs have been first introduced in \cite{plummerwc} and studied extensively since then. Given a graph $G$, the line graph $L(G)$ is the graph obtained by representing every edge of $G$ with a vertex in $L(G)$ and making two vertices of $L(G)$ adjacent if the edges of $G$ represented by these vertices have a common endvertex. It follows that a graph $G$ is equimatchable if and only if its line graph $L(G)$ is well-covered. Motivated by this link and the related research on well-covered graphs, we investigate in this paper the criticality of equimatchable graphs, which has been posed as an open question on well-covered graphs in \cite{LM2016} and reformulated in \cite{ESE} in terms of equimatchable graphs.

 A graph is \emph{1-well-covered} if it is well-covered and remains well-covered upon removal of any vertex \cite{Staples79}. Recently, the stability of being equimatchable with respect to edge removals has been studied in \cite{ESE}. An equimatchable graph $G$ is called \emph{edge-stable} if the graph obtained by the removal of any edge of $G$ remains equimatchable.  So, a graph is edge-stable equimatchable if and only if its line graph is 1-well-covered.  A \emph{shedding vertex} is a vertex $x$  such that for every independent set $I$ in the graph obtained by removing the neighborhood of $x$ and the vertex $x$, there exists some neighbor $y$ of $x$ such that $I\cup \{y\}$ is independent. Shedding vertices are strongly related to the combinatorial topology of independence complexes of graphs \cite{LM2016,wood2009}, and play an important role in identifying vertex decomposable graphs \cite{BC2014}. In \cite{LM2016}, Levit and Mandrescu showed that all vertices of a well-covered graph $G$ without isolated vertices are shedding if and only if $G$ is 1-well-covered, and posed their characterization as an open problem. A partial answer has been given in \cite{ESE} by showing that the characterization of edge-stable equimatchable graphs (with no component isomorphic to an edge) provides a characterization for well-covered line graphs such that all vertices are shedding. In the same paper  \cite{LM2016}, finding all well-covered graphs having no shedding vertex has been posed as an open problem. In terms of equimatchable graphs, this corresponds to the notion of criticality which is the opposite of stability. In this paper, we investigate edge-critical equimatchable graphs which correspond to well-covered line graphs with no shedding vertex; we provide their characterization in some cases and shed light to their structure from various perspectives.

For an equimatchable graph $G$, we say that $e \in E(G)$ is a \emph{critical edge} if the removal of $e$ from $G$ makes it non-equimatchable. Note that if an equimatchable graph $G$ is not edge-stable, then it has a critical edge. A graph $G$ is called \emph{edge-critical equimatchable}, denoted \emph{$\ECE$} for short, if $G$ is equimatchable and every $e \in E(G)$ is critical. We note that $\ECE$-graphs can be obtained from any equimatchable graph by recursively removing non-critical edges. By definition of $\ECE$-graphs, a graph $G$ with no component isomorphic to an edge is $\ECE$ if and only if $L(G)$ is well-covered and has no shedding vertex. Thus, the complete characterization of $\ECE$-graphs would enlighten the structure and the recognition of well-covered line graphs with no shedding vertex.

At the expense of losing the link with 1-well-covered graphs, one can also extend the notion of criticality of equimatchable graphs to vertex removals. An equimatchable graph $G$ is called \emph{vertex-critical} if $G$ looses its equimatchability by the removal of any vertex. We denote vertex-critical equimatchable graphs shortly by $\VCE$.

We start with formal definitions and frequently used results on equimatchability in Section \ref{sec:prem}. We proceed with the characterization of $\VCE$-graphs in Section \ref{sec:VCE}. Our findings point out that apart from an easily detectable simple structure, $\VCE$-graphs coincide with factor-critical equimatchable graphs and that they contain all factor-critical $\ECE$-graphs. This motivates once again the study of $\ECE$-graphs, which we start in Section \ref{sec:ECE}. We first show that $\ECE$-graphs are either 2-connected factor-critical or 2-connected bipartite or even cliques. Noting that 2-connected bipartite $\ECE$-graphs admit a simple characterization, we focus on factor-critical $\ECE$-graphs. We give a complete characterization of $\ECE$-graphs with connectivity 2. In Section \ref{sec:compare}, we provide a comparison of various subclasses of equimatchable graphs in terms of inclusions and intersections; $\ECE$-graphs, $\VCE$-graphs, edge-stable equimatchable graphs and factor-critical equimatchable graphs are illustrated in Figure \ref{fig: ECE VCE}. We conclude in Section \ref{sec:conclusion} with a discussion on factor-critical $\ECE$-graphs with connectivity at least 3.


\section{Definitions and Preliminaries}\label{sec:prem}

Given a graph $G=(V,E)$ and a subset of vertices $I$, $G[I]$ denotes the subgraph of $G$ induced by $I$, and $G\setminus I=G[V\setminus I]$. If $I$ is a singleton $\{v\}$, we denote $G\setminus I$ by $G - v$. We also denote by $G\setminus e$ the graph $G(V,E\setminus\{e\})$. For a subset $I$ of vertices, we say that $I$ is \emph{complete} to another subset $I'$ of vertices (or by abuse of notation, to a subgraph $H$) if all vertices of $I$ are adjacent to all vertices of $I'$ (respectively $H$). $K_r$ is a clique on $r$ vertices. For a vertex $v$, the neighborhood of $v$ in a subgraph $H$ is denoted by $N_H(v)$. We omit the subscript $H$ when it is clear from the context. For a subset $V'\subseteq V$, $N(V')$ is the union of the neighborhoods of the vertices in $V'$. The \textit{degree} of a vertex $v$ is the number of its neighbors, denoted by $d(v)$. For a graph $G$, $\Delta(G)$ denotes the maximum degree of a vertex in $G$. For a connected graph $G$, a $k$-cut set is a set of vertices whose removal disconnects the graph into at least two connected components. A 1-cut set is called a \emph{cut vertex}. The smallest $k$ such that $G$ has a $k$-cut set is called the \emph{connectivity} of $G$. For simplicity, we sometimes abuse the language and use a connected component and the graph induced by this connected component interchangably.

Given a graph $G$, the size of a maximum matching of $G$ is denoted by $\nu(G)$. A matching is \emph{maximal} if no other matching properly contains it. A matching $M$ is said to \emph{saturate} a vertex $v$ if $v$ is an endvertex of some edge in $M$, otherwise it leaves a vertex \emph{exposed}. If every matching $M$ of $G$ \emph{extends} to a perfect matching, in other words, for every matching $M$ (including a single edge) there is a perfect matching that contains $M$, then $G$ is called \emph{randomly matchable}. Clearly, if an equimatchable graph has a perfect matching, then it is randomly matchable. If $G - v$ has a perfect matching for every $v \in V(G)$, then $G$ is called \emph{factor-critical}. For short, a factor-critical equimatchable graph is denoted an \emph{$\EFC$-graph}. For a vertex $v$, a matching $M$ is called a \emph{matching isolating $v$} if $\{v\}$ is a connected component of $G \setminus V(M)$. If $G$ is factor-critical, it follows from its definition that for every vertex $v$, there is a matching $M_v$ isolating $v$. 

The following result serves as a guideline to study the structure of equimatchable graphs.

\begin{thm}\cite{Plummer} \label{Thm: all 2 connected graphs}
A 2-connected equimatchable graph is either factor-critical or bipartite or $K_{2t}$ for some $t\geq 1$.
\end{thm}

In the view of Theorem \ref{Thm: all 2 connected graphs}, a systematic way to study various properties of (subclasses of) equimatchable graphs is to consider i) equimatchable graphs with a cut vertex, ii) 2-connected EFC-graphs, iii) 2-connected bipartite equimatchable graphs, and iv) $K_{2t}$ for some $t\geq 1$.  

The case of bipartite graphs has been settled as follows.

\begin{lem}\cite{Plummer} \label{Lem: chr of equim bip graph}
A connected bipartite graph $G=(U \cup W, E)$, $\vert U \vert \leq \vert W \vert$ is equimatchable if and only if for every $u \in U$, there exists a non-empty set $S\subseteq N(u)$ such that $\vert N(S) \vert \leq \vert S \vert $. 
\end{lem}

Lemma \ref{Lem: chr of equim bip graph}, together with the well-known Hall's condition implies the following more insightful characterization of connected bipartite equimatchable graphs.

\begin{thm}[Hall's Theorem]\cite{hall}\label{thm:hall}
A bipartite graph $G = (A\cup B,E)$ has a matching saturating
all vertices in $A$ if and only if it satisfies $|N(S)| \geq |S|$ for every subset $S \subseteq A$.
\end{thm}

\begin{cor}\cite{ESE}\label{cor:bip-saturate-U}
Let $G = (U \cup W, E)$ be a connected bipartite graph with $|U| \leq |W|$. Then $G$ is equimatchable if and only if every maximal matching of $G$ saturates $U$.
\end{cor}

While studying equimatchable graphs with  a cut vertex, the following will be useful:

\begin{lem}\cite{eqm_regular}\label{lem:cut vertex equim}
Let $G$ be a connected equimatchable graph with a cut vertex $v$, then each connected component of $G-v$ is also equimatchable.
\end{lem}

It should be noted that in the studies of equimatchable graphs with respect to various properties in \cite{EK2, Kotbic, Favaron}, the case of factor-critical equimatchable graphs has been the most complicated one. The following basic observations will guide us through our proofs. Since the size of any maximal matching in a factor-critical equimatchable graph is $(n-1)/2$ where $n$ is the number of vertices of the graph, we have the following:
\begin{lem} \cite{ESE} \label{lem: defn equim}
Let $G$ be a factor-critical graph. $G$ is equimatchable if and only if there is no independent set $I$ with 3 vertices such that $G\setminus I$ has a perfect matching.
\end{lem}
An equivalent reformulation of Lemma \ref{lem: defn equim} is the following:
\begin{cor}\label{cor: def EFC}
Let $G$ be a factor-critical equimatchable graph. Then every maximal matching of $G$ leaves exactly one vertex exposed. 
\end{cor}

Another useful result on factor-critical equimatchable graphs  is the following. 
\begin{lem}\cite{EK2}\label{lem: gen isolating matching}
Let $G$ be a 2-connected factor-critical equimatchable graph. Let $v$ be a vertex of $G$ and $M_v$ a minimal matching isolating $v$. Then  $G\setminus (V(M_v)\cup \{v\})$ is isomorphic to $K_{2n}$ or $K_{n,n}$ for some $n\in \N$.
\end{lem}

Lastly, equimatchable graphs with a perfect matching are precisely randomly matchable graphs whose structure is well-known:

\begin{lem}
\label{lem:randomlymatchable}\cite{Sumner}
A connected graph is randomly matchable if and only if it is isomorphic to a $K_{2n}$ or a $K_{n,n} \ (n\geq1)$.
\end{lem}

\section{Vertex-critical Equimatchable graphs}\label{sec:VCE}

Let us first investigate vertex-critical equimatchable graphs. As suggested by Theorem \ref{Thm: all 2 connected graphs}, we will proceed seperately with $\VCE$-graphs with a cut vertex, 2-connected bipartite $\VCE$-graphs, even cliques (showing that all three of them are empty), and finally with 2-connected factor-critical $\VCE$-graphs. As a result, we will show that $\VCE$-graphs are almost equivalent to factor-critical equimatchable graphs. Building upon the results obtained in this section, we will show later that $\VCE$-graphs contain factor-critical $\ECE$-graphs. This motivates even further the study of factor-critical $\ECE$-graphs.

Recall that a graph $G$ is $\VCE$  if $G$ is equimatchable and $G-v$ is non-equimatchable for every $v \in V(G)$. 
Let us call a vertex $v \in V(G)$ \emph{strong} (in $G$) if every maximal matching of $G$ saturates $v$, (or equivalently there is no maximal matching of $G -v$ saturating all neighbours of $v$), otherwise it is called \emph{weak} (in $G$).

We have the following by noticing that every maximal matching of $G$ saturates $v$ if and only if the size of every maximal matching of $G$ decreases exactly by one when $v$ is removed from $G$:

\begin{rem}\label{rem: strong cut vertex}
Let $G$ be an equimatchable graph. Then $v$ is a strong vertex if and only if $\nu(G-v)=\nu(G)-1$.
\end{rem}

\begin{prop}\label{prop: when $G-v$ is equim.}
Let $G$ be an equimatchable graph. Then, for a vertex $v\in V(G)$, the graph $G-v$ is equimatchable if and only if one of the following holds:
\begin{itemize}
\item[$(i)$] $v$ is a strong vertex in $G$,
\item[$(ii)$] all vertices in $N(v)$ are strong in $G-v$.  
\end{itemize} 
\end{prop}
\begin{proof}
Let $G$ be an equimatchable graph, and assume that $G-v$ is equimatchable for a vertex $v\in V(G)$. There are two possibilities: Either $\nu(G-v)=\nu(G)-1$ and then  $v$ is a strong vertex by Remark \ref{rem: strong cut vertex}. Otherwise, $\nu(G-v)=\nu(G)$, i.e., $v$ is not strong, hence it is a weak vertex.  Then we claim that $N(v)$ is a set of strong vertices in $G-v$. Indeed, if $u \in N(v)$ is a weak vertex in $G-v$, then there exists a maximal matching $M$ of $G-v$ leaving $u$ exposed with $\vert M \vert=\nu(G-v)=\nu(G)$. Then $M \cup \{vu\}$ is a maximal matching in $G$, a contradiction with the equimatchability of $G$. Hence $N(v)$ is a set of strong vertices in $G-v$.

We now suppose the converse. Let $G$ be an equimatchable graph, and let $v$ be a strong vertex. Then $\nu(G-v)=\nu(G)-1$ and the size of each maximal matching decreases exactly by one. It follows that $G-v$ is equimatchable.  Now, let $N(v)$ be a set of strong vertices in $G-v$. Then every maximal matching of $G-v$ saturates $N(v)$, and therefore those are also maximal matchings of $G$. Since $G$ is equimatchable, they all have the same size, thus $G-v$ is also equimatchable.
\end{proof}

By Lemma \ref{lem:cut vertex equim}, if an equimatchable graph has a cut vertex, then its removal from the graph leaves an equimatchable graph. Then we have the following.

\begin{prop}\label{prop:VCE-graphs are 2-connected}
$\VCE$-graphs are 2-connected.
\end{prop}

\begin{prop}\label{prop: there is no bip VCE graph}
There is no bipartite $\VCE$-graph.
\end{prop}
\begin{proof}
Let $G=(U \cup W, E)$ be a bipartite equimatchable graph  with $\vert U \vert \leq \vert W \vert$. By Corollary \ref{cor:bip-saturate-U}, every vertex of $U$ is strong. It follows that  $G-u$ is equimatchable  for every $u \in U$ by Proposition \ref{prop: when $G-v$ is equim.}. Therefore, $G$ is not VCE. Hence, there is no a bipartite VCE-graph. 
\end{proof}

Since every complete graph is an equimatchable graph, the removal of a vertex from $K_{t}$ yields an equimatchable graph. Thus we have the following.

\begin{prop}\label{prop:K2t-is-not-VCE}
$K_{t}$ for some integer $t\geq 2 $ is not $\VCE$.
\end{prop}

Theorem \ref{Thm: all 2 connected graphs} together with Propositions \ref{prop:VCE-graphs are 2-connected}, \ref{prop: there is no bip VCE graph} and  \ref{prop:K2t-is-not-VCE} imply the following:

\begin{cor}
VCE-graphs are 2-connected factor-critical.
\end{cor}

 So, the following result provides a characterization of all VCE-graphs.

\begin{thm}\label{thm:main-vce}
Let $G$ be a 2-connected graph with $2r+1$ vertices. Then $G$ is $\VCE$ if and only if $G$ is a $(K_{2r},K_{r,r})$-free $\EFC$-graph.
\end{thm}
\begin{proof}
Let $G$ be a $\VCE$-graph with $2r+1$ vertices, then it is 2-connected by Proposition \ref{prop:VCE-graphs are 2-connected}. We claim that $G$ is $(K_{2r},K_{r,r})$-free, since otherwise there exists a vertex $v \in V(G)$ such that $G-v$ is isomorphic to a connected randomly matchable graph. In such a case $G-v$ is equimatchable, a contradiction with the vertex critically of $G$. 

We now suppose the converse. Let $G$ be a $(K_{2r},K_{r,r})$-free EFC-graph. Assume for a contradiction that there is a vertex $v \in V(G)$ such that $G-v$ is equimatchable. Since $G$ is  factor-critical, the graph $G-v$ has a perfect matching.  It follows that $G-v$ is a connected randomly matchable graph which is either $K_{2r}$ or $K_{r,r}$ by Lemma \ref{lem:randomlymatchable}, contradicting to our assumption. Therefore $G$ is $\VCE$.   
\end{proof}

Having obtained a characterization of VCE-graphs as a subclass of EFC-graphs, let us now investigate the difference of EFC-graphs from VCE-graphs. This will allow us to complete the containment relationships between various subclasses of equimatchable graphs as depicted in  Figure \ref{fig: ECE VCE} of Section \ref{sec:compare}.

\begin{thm}\cite{Favaron}\label{thm:vce-ece}
$G$ is an EFC-graph with a cut-vertex $v$ if and only if every connected component $C_i$ of $G-v$ is isomorphic to $K_{r,r}$ or to $K_{2t}$ for some integers $r,t\geq 1$ and where $v$ is adjacent to at least two adjacent vertices of each $C_i$. 
\end{thm}

Theorem \ref{thm:main-vce} together with Propositions \ref{prop:VCE-graphs are 2-connected} and Theorem \ref{thm:vce-ece} allow us to describe all $\EFC$-graphs that are not $\VCE$ as follows:

\begin{prop}\label{prop:EFC-notVCE}
Let $G$ be an EFC-graph which is not VCE. Then there exists a vertex $v\in V(G)$ such that each connected component $C_i$ of $G-v$ is a $K_{r,r}$ or a $K_{2t}$ for some integers $r,t\geq 1$ and where $v$ is adjacent to at least two adjacent vertices of each $C_i$. 
\end{prop}
\begin{proof}
If $G$ has a cut-vertex, then the result follows from Theorem \ref{thm:vce-ece}. Otherwise $G$ is a 2-connected graph with $2r+1$ vertices and contains one of $K_{2r}$ or $K_{r,r}$ by Theorem \ref{thm:main-vce}. Since $G$ is 2-connected, $v$ has at least two neighbors $x$ and $y$ in $G-v$; moreover $xy\in E$. Indeed, if $G-v$  is $K_{2r}$ then clearly $xy\in E$; if $G-v$ is $K_{r,r}$ then $x$ and $y$ belong to the same $(r)$-stable set of $K_{r,r}$ and $v$ has no neighbor in the other $(r)$-stable set; then $G-x$ has no perfect matching, contradicting that $G$ is factor-critical. 
\end{proof}

It follows from the above discussion that $\VCE$-graphs are almost equivalent to the class of factor-critical equimatchable graphs; indeed this is the most intriguing subclass of equimatchable graphs  as  the structure of the remaining equimatchable graphs are rather well-known \cite{Plummer, ECEbip}. By Proposition \ref{prop:EFC-notVCE}, the only factor-critical equimatchable graphs that are not $\VCE$ are those graphs $G$ admitting a vertex $v$ such that $G- v$ leaves a graph whose connected components are $K_{r,r}$ or $K_{2t}$  for some integers $r$ and $t$  and where $v$ is adjacent to at least two adjacent vertices of each component of $G-v$.

 We now start the investigation of $\ECE$-graphs. It is worth noting that while comparing subclasses of equimatchable graphs in Section \ref{sec:compare}, Proposition \ref{prop:VCE-graphs are 2-connected} and Theorem \ref{thm:main-vce} will allow us to derive (in Corollary \ref{cor:ece-vce}) that all factor-critical $\ECE$-graphs are $\VCE$.



\section{Edge-critical Equimatchable graphs}\label{sec:ECE}

In this section, we investigate ECE-graphs.  Our preliminary results in Section \ref{sec:premECE} show that apart from two simple cases, namely bipartite ECE-graphs and complete graphs of even order, all ECE-graphs are (2-connected) factor-critical. Then, we characterize factor-critical ECE-graphs with connectivity 2 in Section \ref{sec:fcECE} (Theorem \ref{thm:chrc of ECE conn 2}). 

\subsection{Preliminaries on ECE-graphs} \label{sec:premECE}~~\medskip

We start with a lemma that will be frequently used in our proofs.

\begin{lem}\label{lem: critical edge iff there is a matching}
Let $G$ be an equimatchable graph. Then $uv \in E(G)$ is critical if and only if there is a matching of $G$ containing $uv$ and saturating $N(\{u,v\})$. 
\end{lem}
\begin{proof}
Assume that $uv$ is a critical edge in $G$. Then $G \sm uv$ admits two maximal matchings $M_1$ and $M_2$ with $|M_1|<|M_2|$. Note that  $M_1$ leaves both $u$ and $v$ exposed in $G \sm uv$ since otherwise $M_1$ would be a maximal matching of $G$, contradicting that $G$ is equimatchable.
This implies that $M_1$ saturates all vertices in $N_{G\sm uv}(\{u,v\})$. Hence, $M_1 \cup \{uv\}$ is a maximal matching of $G$ as desired.

We now suppose the converse. If there is such a matching $M$, then $M\sm\{uv\}$ is a maximal matching in $G\sm uv$. However, there is also another maximal matching $M'$ in $G\sm uv$ which can be obtained by extending the edge $vw$ for some $w\in N(v)$ with $w \neq u$. Clearly, $M'$ has size $\nu(G)$ since it is also a maximal matching of $G$. Then $G\sm uv$ is not equimatchable, implying that $uv$ is a critical-edge in $G$.
\end{proof}

By Lemma  \ref{lem: critical edge iff there is a matching}, if a graph $G$ is factor-critical ECE, then for every $uv \in E(G)$, there exists a matching of $G$ containing $uv$ and saturating $N(\{u,v\})$. However such a matching does not exist if $N(\{u,v\})= V(G)$ since $|G|$ is odd. It follows that:

\begin{cor}\label{cor:no-dom-edge}
If $G$ is a connected factor-critical $\ECE$-graph, then there is no edge $uv\in E(G)$ such that $N(\{u,v\})= V(G)$.
\end{cor}


%
%
%
%

The following is a direct consequence of Lemma \ref{lem: critical edge iff there is a matching}, since any randomly matchable graph has a matching containing $uv$ and saturating $N(\{u,v\})$ for every edge $uv$.

\begin{rem}
Randomly matchable graphs are edge-critical equimatchable.
\end{rem}

In what follows, we shall prove that $\ECE$-graphs have no cut vertex. 

\begin{lem}\label{lem:ECE 2 connected}
$\ECE$-graphs are 2-connected.
\end{lem}
\begin{proof}
Assume that $G$ is an $\ECE$-graph, and has a cut-vertex $z$. Let $H_1,H_2,\ldots, H_k$ be the components of $G-z$ for $k\geq 2$. By Lemma \ref{lem:cut vertex equim}, each $H_i$ is equimatchable. Besides, since $G$ is ECE-graph, there exists a matching containing $uv$ and saturating $N(\{u,v\})$ for every $uv \in E(G)$ by Lemma \ref{lem: critical edge iff there is a matching}. Let us pick a vertex $w_i $ from each component $H_i$ such that $w_i\in N(z)\cap H_i$. Then, for the edge $zw_1$,  there exists a maximal matching $M$ in $G$ containing $zw_1$ and saturating $N(\{z,w_1\})$.  Let $M\cap E(H_i)=M_i$ for $i\in  [k] $. 
Observe that $M_1$ saturates all vertices in $N_{H_1}(w_1)$. Also,  for each $i\geq 2$, the matching $M_i$ saturates all vertices in $N_{H_i}(z)$. 
We now consider the edge $zw_2$, similarly as above; there exists a maximal matching $L$ in $G$ containing $zw_2$ and saturating $N(\{z,w_2\})$. It follows that there exists a maximal matching $T=L\cap E(H_1)$ in $H_1$ such that $T$ saturates all vertices in $N_{H_1}(z)$.
In this manner, we obtain a maximal matching $T\cup M_2\cup \ldots \cup  M_k$ in $G$ isolating $z$, and so we have $\nu(G)=|T|+|M_2|+\ldots+|M_k|$ since $G$ is equimatchable. This also implies that $\nu(G)=\sum \nu(H_i)$. On the other hand, observe that $M_1$ is a maximal matching in $H_1$ since $M_1$ saturates all vertices in $N_{H_1}(w_1)$. Moreover the matchings $M_1$ and $T$ are of the same size since $H_1$ is equimatchable. Thus $M_1\cup M_2\cup \ldots \cup  M_k$ must be of size $\nu(G)$. However, this contradicts that we have the maximal matching $M=M_1\cup M_2\cup \ldots \cup  M_k\cup \{zw_1\}$ in $G$. Hence $G$ has no cut-vertex.
\end{proof}

The following is an immediate consequence of Theorem \ref{Thm: all 2 connected graphs} together with Lemma \ref{lem:ECE 2 connected}.
\begin{thm}\label{thm:ece-categories}
$\ECE$-graphs are either factor-critical or bipartite or $K_{2t}$ for some $t\geq 1$.
\end{thm}

In view of Theorem \ref{thm:ece-categories}, we consider ECE-graphs under three disjoint categories: 2-connected factor-critical, 2-connected bipartite, and complete graphs of even order (which are randomly matchable thus ECE).

The characterization of bipartite ECE-graphs has been given in \cite{ECEbip} as follows. 

\begin{thm}\cite{ECEbip}\label{prop: charc of ECE bip graph}
A connected bipartite graph $G = (U \cup V,E)$ with $|U| \leq |V |$ except $K_2$ is an bipartite $\ECE$-graph if and only if for every $u \in U, |N(S)| \geq  |S|$ holds
for any subset $S \subseteq N(u)$ and the equality holds only for $S = N(u)$.
\end{thm}

It remains to enlighten the structure of factor-critical ECE-graphs. Recall that all factor-critical ECE-graphs are 2-connected by Lemma \ref{lem:ECE 2 connected}. In the next subsection, we provide a characterization of factor-critical ECE-graphs with connectivity 2.

\subsection{Factor-critical ECE-graphs with connectivity 2} \label{sec:fcECE}~~\medskip

The following result on factor-critical equimatchable graphs with connectivity 2 will guide us in this subsection.  

\begin{thm}\cite{Favaron} \label{thm: Favaron connectivity 2 }
Let $G$ be an $\EFC$-graph of order at least 5 and connectivity 2 and let $S=\{s_1,s_2\}$ be 2-vertex-cut of G. Suppose that $a_i$ and $b_i$ are distinct neighbours of $s_i$ in respectively $A$ and $B$ for $i=1,2$.  Then $G - S$ has precisely two components A and B such that 
\begin{itemize}
\item[$(i)$] B is one of the four graphs $K_{2p+1},K_{2p+1} \sm b_1b_2, \ K_{p,p+1}$ or $K_{p,p+1}+b_1b_2$ and in the two last cases $b_1$ and $b_2$ belong to the $(p+1)$-stable set of $K_{p,p+1}$.
\item[$(ii)$] $A- \{a_1,a_2\}$ is either $K_{2q-2}$ or $K_{q-1,q-1}$, and if $\vert B \vert > 1$, then A is either $K_{2q}$ or $K_{q,q}$ 
\end{itemize}
\end{thm}

First, we show that there is no factor-critical ECE-graph of order $5$ or less.

\begin{rem}\label{rem:7 vertices}
Factor-critical $\ECE$-graphs have at least 7 vertices.
\end{rem}
\begin{proof}
It is clear that there is no ECE-graph on $3$ or less vertices. Assume that there exists a connected factor-critical ECE-graph $G$ with 5 vertices. If $\Delta(G)=2$, then for any edge $uv \in E(G)$ with $d(u)=d(v)=2$, there is no matching containing $uv$ and saturating $N(\{u,v\})$. Thus, $uv$ is not critical by Lemma \ref{lem: critical edge iff there is a matching}. For the other case, if  $\Delta(G) \geq 3$, then for a vertex $u$ with $d(u)\geq 3$, there exists a neighbour $v$ of $u$ such that $N[\{u,v\}]=V(G)$, a contradiction by Corollary \ref{cor:no-dom-edge}.   
\end{proof}

 \medskip

The general structure of a factor-critical ECE-graph $G$ of order at least $7$ and connectivity $2$ follows from Theorem \ref{thm: Favaron connectivity 2 }; for a 2-cut $S=\{s_1,s_2\}$, the graph $G -S$ has exactly two components $A$ and $B$ as described in Theorem \ref{thm: Favaron connectivity 2 } and illustrated in Figure \ref{fig:FCE graph}. 

\begin{figure}[htb]
\centering 
\begin{tikzpicture}[scale=.8]
\node [noddee] at (0,1) (a1) [label=above: \scriptsize $a_1$]  {};
\node [noddee] at (0,0) (a2) [label=below: \scriptsize $a_2$]  {}
	edge [] (a1);

\node [noddee] at (4,1) (b1) [label=above: \scriptsize $b_1$]  {};
\node [noddee] at (4,0) (b2) [label=below: \scriptsize $b_2$]  {};	
\node [noddee] at (5,.5) (b3)   {}
	edge [] (b1)
	edge [] (b2);
\node [noddee] at (2,1) (s1) [label=above: \scriptsize $s_1$]  {}
	edge [] (b1)
	edge [] (a1);
\node [noddee] at (2,0) (s2) [label=below: \scriptsize $s_2$]  {}
	edge [] (b2)
	edge [] (a2);

\draw[thick] (-.3,.5) ellipse (1.4cm and 1.4cm);		
\draw[thick] (4.3,.5) ellipse (1.4cm and 1.4cm);		

\node at (-.5,-1.3) {$A$};
\node at (4.5,-1.3) {$B$};	
\end{tikzpicture}  
\caption{The structure of factor-critical ECE-graphs where $|A|$ is even and $|B|$ is odd.}
\label{fig:FCE graph}
\end{figure}
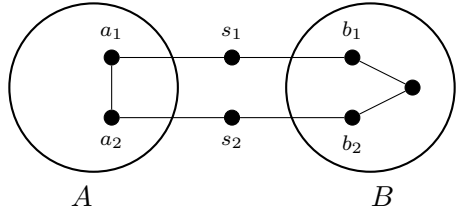

We will introduce five possible configurations with respect to $A$ and $B$ and then show that a factor-critical ECE-graph with connectivity 2 falls into one of these five types.  For a graph $G$, consider $A,B \subset V(G)$ each one with at least two vertices. We say that $A$ is \emph{partially-complete} to $B$ 
if there exist a non-empty partition $A_1,A_2$  of $A$ and a non-empty partition $B_1,B_2$  of $B$  such that for each $i=1,2$, $A_i$ is complete to $B_i$ and $A_i$ has no neighbour in $B_{3-i}$. Let $S=\{s_1,s_2\}$ be an independent set. \medskip

\begin{itemize}
\item \textbf{Type I}: $A\cong K_{2q}$, $B \cong K_{2p+1}$ for $p,q\geq 1$  such that  $S$ is complete to $A$, and $S$ is partially-complete to $B$ (see Figure \ref{fig: FC- ESE graph (a)}). \medskip	
\item \textbf{Type II}: $A\cong K_{q,q}$, $B \cong K_{2p+1}$ for $p,q\geq 1$  such that  for $i=1,2$, each $s_i$ is complete to a distinct $(q)$-stable set of $A$, and $S$ is partially-complete to $B$ (see Figure \ref{fig: FC- ESE graph (b)}). 	\medskip
\item \textbf{Type III}:   $A\cong K_{2q}$, $B \cong K_{p,p+1}$ for $p,q\geq 1$ such that  $S$ is complete to $A$, and $S$ is partially-complete to the $(p+1)$-stable set of $B$ (see Figure \ref{fig: FC- ESE graph (c)}). \medskip
\item \textbf{Type IV}:  $A\cong K_{q,q}$, $B \cong K_{p,p+1}$ for $p,q\geq 1$ such that  for $i=1,2$, each $s_i$  is complete to distinct $(q)$-stable sets of $A$, and $S$ is partially-complete to the $(p+1)$-stable set of $B$ (see Figure \ref{fig: FC- ESE graph (d)}). \medskip
\item \textbf{Type V}:  $A - \{a_1, a_2\} \cong K_{2q-2}$ for $q\geq 3$, $B \cong K_1$, and there is a vertex $w \in A$ such that $\{a_1, a_2,w\}$ is a stable set, each $s_i$ is complete to $\{b, a_i,w\}$ for $i = 1, 2$, and
$\{a_1, a_2,w\}$ is complete to $A-\{a_1, a_2,w\}$ (see Figure \ref{fig: FC- ESE graph (e)}).
\end{itemize}


\begin{figure}[htb]
\centering     
\subfigure[Type I]{\label{fig: FC- ESE graph (a)}
\begin{tikzpicture}[scale=.6]
\node [noddee1] at (0,0) (a1)  {};
\node [noddee1] at (0,2) (a2)  {}
	edge [] (a1);
\node [noddee1] at (2,0) (a3)  {}
	edge [] (a1)
	edge [] (a2);
\node [noddee1] at (2,2) (a4)  {}
	edge [] (a1)
	edge [] (a2)
	edge [] (a3);
\node [noddee1] at (6,1) (b1)  {};
\node [noddee1] at (7,2) (b2)  {}
	edge [] (b1);
\node [noddee1] at (7,0) (b3)  {}
	edge [] (b1)
	edge [] (b2);
\node [noddee1] at (4,1.5) (s1) [label=above: \scriptsize $s_1$]  {}
	edge [] (a1)
	edge [] (a2)
	edge [] (a3)
	edge [] (a4)
	edge [] (b1)
	edge [] (b2);
\node [noddee1] at (4,.5) (s2) [label=below: \scriptsize $s_2$]  {}
	edge [] (a1)
	edge [] (a2)
	edge [] (a3)
	edge [] (a4)
	edge [] (b3);
\node at (1,-1) {\scriptsize $A\cong K_{2q}$};
\node at (7,-1) {\scriptsize $B \cong K_{2p+1}$};	
\end{tikzpicture} }\hspace*{1cm}
\subfigure[Type II]{\label{fig: FC- ESE graph (b)}
\begin{tikzpicture}[scale=.6]
\node [noddee1] at (0,0) (a1)  {};
\node [noddee1] at (0,2) (a2)  {}
	edge [] (a1);
\node [noddee1] at (2,0) (a3)  {}
	edge [] (a2);
\node [noddee1] at (2,2) (a4)  {}
	edge [] (a1)
	edge [] (a3);
\node [noddee1] at (6,1) (b1)  {};
\node [noddee1] at (7,2) (b2)  {}
	edge [] (b1);
\node [noddee1] at (7,0) (b3)  {}
	edge [] (b1)
	edge [] (b2);
\node [noddee1] at (4,1.5) (s1) [label=above: \scriptsize $s_1$]  {}
	edge [] (a2)
	edge [] (a4)
	edge [] (b1)
	edge [] (b2);
\node [noddee1] at (4,.5) (s2) [label=below: \scriptsize $s_2$]  {}
	edge [] (a1)
	edge [] (a3)
	edge [] (b3);
\node at (1,-1) {\scriptsize $A\cong K_{q,q}$};
\node at (7,-1) {\scriptsize $B \cong K_{2p+1}$};		
\end{tikzpicture} }
\subfigure[Type III]{\label{fig: FC- ESE graph (c)}
\begin{tikzpicture}[scale=.6]
\node [noddee1] at (0,0) (a1)  {};
\node [noddee1] at (0,2) (a2)  {}
	edge [] (a1);
\node [noddee1] at (2,0) (a3)  {}
	edge [] (a1)
	edge [] (a2);
\node [noddee1] at (2,2) (a4)  {}
	edge [] (a1)
	edge [] (a2)
	edge [] (a3);
\node [noddee1] at (6,0) (b1)  {};
\node [noddee1] at (6,1) (b2)  {};
\node [noddee1] at (6,2) (b3)  {};
\node [noddee1] at (7.5,.5) (b11)  {}
	edge [] (b1)
	edge [] (b2)
	edge [] (b3);
\node [noddee1] at (7.5,1.5) (b11)  {}
	edge [] (b1)
	edge [] (b2)
	edge [] (b3);
\node [noddee1] at (4,1.5) (s1) [label=above: \scriptsize $s_1$]  {}
	edge [] (a1)
	edge [] (a2)
	edge [] (a3)
	edge [] (a4)
	edge [] (b3);
\node [noddee1] at (4,.5) (s2) [label=below: \scriptsize $s_2$]  {}
	edge [] (a1)
	edge [] (a2)
	edge [] (a3)
	edge [] (a4)
	edge [] (b1)
	edge [] (b2);
\node at (1,-1) {\scriptsize $A\cong K_{2q}$};
\node at (7,-1) {\scriptsize $B \cong K_{p,p+1}$};	
\end{tikzpicture} }\hspace*{1cm}
\subfigure[Type IV]{\label{fig: FC- ESE graph (d)}
\begin{tikzpicture}[scale=.6]
\node [noddee1] at (0,0) (a1)  {};
\node [noddee1] at (0,2) (a2)  {}
	edge [] (a1);
\node [noddee1] at (2,0) (a3)  {}
	edge [] (a2);
\node [noddee1] at (2,2) (a4)  {}
	edge [] (a1)
	edge [] (a3);
\node [noddee1] at (6,0) (b1)  {};
\node [noddee1] at (6,1) (b2)  {};
\node [noddee1] at (6,2) (b3)  {};
\node [noddee1] at (7.5,.5) (b11)  {}
	edge [] (b1)
	edge [] (b2)
	edge [] (b3);
\node [noddee1] at (7.5,1.5) (b11)  {}
	edge [] (b1)
	edge [] (b2)
	edge [] (b3);
\node [noddee1] at (4,1.4) (s1) [label=above: \scriptsize $s_1$]  {}
	edge [] (a2)
	edge [] (a4)
	edge [] (b3);
\node [noddee1] at (4,.6) (s2) [label=below: \scriptsize $s_2$]  {}
	edge [] (a1)
	edge [] (a3)
	edge [] (b1)
	edge [] (b2);
\node at (1,-1) {\scriptsize $A\cong K_{q,q}$};
\node at (7,-1) {\scriptsize $B \cong K_{p,p+1}$};		
\end{tikzpicture} }

\subfigure[Type V]{\label{fig: FC- ESE graph (e)}
\begin{tikzpicture}[scale=.8]
\node [noddee] at (6,.5) (b) [label=right: \scriptsize $b$]  {};
\node [noddee] at (4.5,1) (s1) [label=above: \scriptsize $s_1$]  {}
	edge [] (b);
\node [noddee] at (4.5,0) (s2) [label=below: \scriptsize $s_2$]  {}
	edge [] (b);
\node [noddee2] at (-.2,.7) (x1) {};
\node [noddee2] at (-.4,1.4) (x2) {}
	edge [] (x1);
\node [noddee2] at (.4,1.2) (x3) {}
	edge [] (x1)
	edge [] (x2);
\node [noddee2] at (.6,.1) (x4) {}
	edge [] (x1)
	edge [] (x2)
	edge [] (x3);
\node [noddee2] at (.2,-.4) (x5) {}
	edge [] (x1)
	edge [] (x2)
	edge [] (x3)
	edge [] (x4);
\node [noddee2] at (-.6,.2) (x6) {}
	edge [] (x1)
	edge [] (x2)
	edge [] (x3)
	edge [] (x4)
	edge [] (x5);
\node [noddee2] at (.7,.7) (x7) {}
	edge [] (x1)
	edge [] (x2)
	edge [] (x3)
	edge [] (x4)
	edge [] (x5)
	edge [] (x6);
\node [noddee] at (3,1.5) (a1) [label=above: \scriptsize $a_1$]  {}
	edge [] (s1)
	edge [] (x1)
	edge [] (x2)
	edge [] (x3)
	edge [] (x4)
	edge [] (x5)
	edge [] (x6)
	edge [] (x7);
\node [noddee] at (3,.5) (w) [label=above: \scriptsize $w$]  {}
	edge [] (s1)
	edge [] (s2)
	edge [] (x1)
	edge [] (x2)
	edge [] (x3)
	edge [] (x4)
	edge [] (x5)
	edge [] (x6)
	edge [] (x7);
\node [noddee] at (3,-.5) (a2) [label=below: \scriptsize $a_2$]  {}
	edge [] (s2)
	edge [] (x1)
	edge [] (x2)
	edge [] (x3)
	edge [] (x4)
	edge [] (x5)
	edge [] (x6)
	edge [] (x7);
\draw[thick] (0,.5) ellipse (1.2cm and 1.4cm);		
\node at (0,-1.3) {$K_{2q-3}$};
\end{tikzpicture} }
\caption{The family $\F$ of factor-critical $\ECE$-graphs with connectivity 2.}
\label{fig:FC- ESE conn 2}
\end{figure}
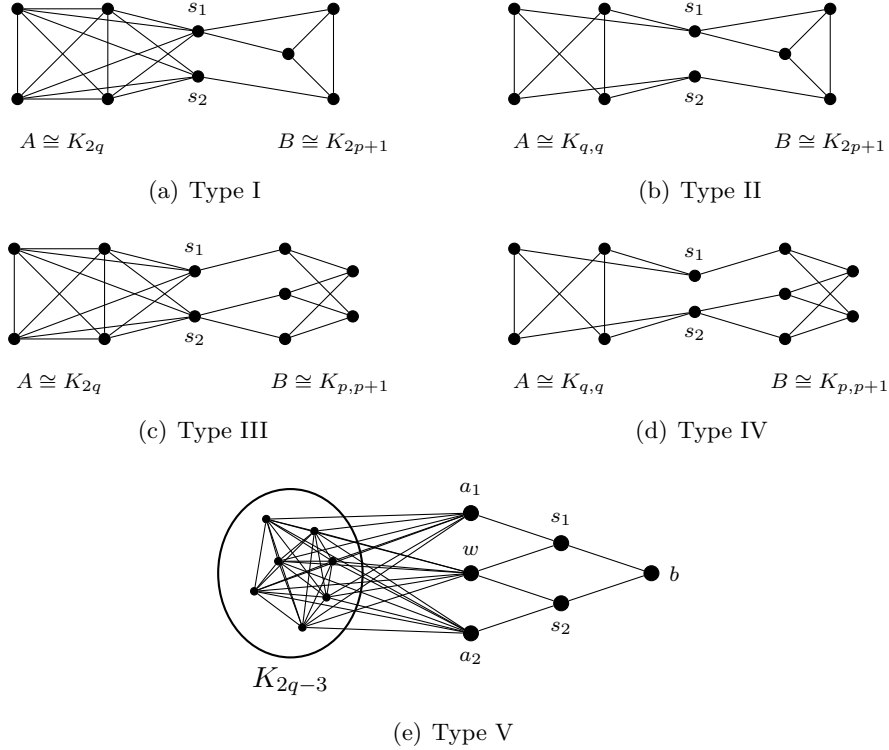

Let $\F$ stand for the family of all graphs falling into one of the five types of configurations I, II, III, IV, V depicted in Figure \ref{fig:FC- ESE conn 2}.
The main result of this section is the following.

\begin{thm}\label{thm:chrc of ECE conn 2}
A graph $G$ is factor-critical $\ECE$ with connectivity $2$ if and only if $G\in \F$ where the five types describing $\F$ are depicted in Figure \ref{fig:FC- ESE conn 2}. \medskip
\end{thm}

We will obtain the proof of Theorem \ref{thm:chrc of ECE conn 2} as two separate lemmas, each one proving one direction.

\begin{lem}\label{lem:All graphs belonging to F are ECE}
All graphs belonging to $\F$ are $\ECE$-graphs. 
\end{lem}

\begin{proof}
Let $G\in \F$. We will show that each one of the five types of configuration is an ECE-graph. Let $\mathcal{G}_i$ be the class of all graphs of Type $i$ as described above and depicted in Figure \ref{fig:FC- ESE conn 2}.

\textbf{$G \in \mathcal G_1$ is equimatchable:} Since $S$ is complete to $A$ and partially-complete to $B$, and since $G- S$ consists of two cliques, it follows that the independence number of $G$ is equal to $2$. Thus $G$ is equimatchable by Lemma \ref{lem: defn equim}. 

\textbf{$G \in \mathcal G_2$ is equimatchable:} Let $A_1,A_2$ be the bipartition of $A$ with  $A_1=\{a_1,a_2,\ldots,a_{q}\}$ and $A_2=\{a_1',a_2',\ldots,a_{q}'\}$ and without loss of generality suppose that for $i=1,2$,  each $s_i$ is complete to $A_i$. Consider an independent set $I$ of size $3$. It is clear that $I$ cannot contain both $s_1$ and $s_2$ since $N(\{s_1,s_2\}) = V(G)$. Thus, there are two possibilities. Either $I$ contains at least two vertices of $A$, say $a_1,a_2 \in A_1$, then $G - I$ has no perfect matching since $\{a_1',a_2',\ldots,a_{q}'\}$ is an independent set in $G - I$ having only $q-1$ neighbours in $G - I$.  Or $I$ contains one vertex from each one of the sets $A,S$ and $B$, say $I=\{a,s,b\}$, respectively. Note that if $s=s_1$, then $a \in A_2$. In this case, $G - I$ has no perfect matching since $\{a_1,a_2,\ldots,a_{q}\}$ is an independent set in $G - I$ having only $q-1$ neighbours in $G - I$. It then follows from Lemma \ref{lem: defn equim} that $G$ is equimatchable. 

\textbf{$G \in \mathcal G_3$ is equimatchable:}
Let $B_1,B_2$ be the bipartition of $B$ with $B_1=\{b_1,b_2,\ldots,b_{p+1}\}$ and $B_2=\{b_1',b_2',\ldots,b_{p}'\}$. Note that $S$ is complete to $A$ and partially-complete to $B_1$.  Consider an independent set $I$ of size $3$. Suppose $I$ contains both vertices of $S=\{s_1,s_2\}$, then its third vertex belongs to $B_2$. We then observe that $G - I$ has no perfect matching since $B_1$ is an independent set of $G-I$ of size $p+1$ but it has only $p-1$ neighbours in $G-I$.  So assume that $I$ contains at most one vertex from $A\cup S$, thus at least two vertices  $x,y$ in $B$, necessarily both in $B_1$ or both in $B_2$. If $x,y \in B_1$, then  $G - I$ has no perfect matching, since $B_2$ is an independent set in $G - I$ of size $p$ but it has at most  $p-1$ neighbours in $G-I$. Finally, if $x,y \in B_2$, then  $G -I$ has no perfect matching since $B_1$ is an independent set in $G -I$ of size $p+1$ but it has at most $p$ neighbours in $G-I$. Hence $G$ is equimatchable by Lemma \ref{lem: defn equim}.

\textbf{$G \in \mathcal G_4$ is equimatchable:} Let $A_1$ and $A_2$ be the bipartition of $A$ with $|A_1|=|A_2|=q\geq 1$ and $s_i$ is complete to $A_i$ for $i=1,2$. Consider an independent set $I$ of size $3$. If $I$ contains both vertices of $S=\{s_1,s_2\}$ or at least two vertices of $B$,  then we can show that $G$ is equimatchable in a similar way as above. Thus, assume that  $I$ contains at least two vertices of $A$ which should be clearly in the same part of $A$, say $A_1$. Then $G - I$ has no perfect matching since $A_2$ is an independent set in $G - I$ of size $q$ but it has at most  $q-1$ neighbours in $G-I$ (in $A_1\cup\{s_2\}$). Finally, if $I$  consists of one vertex from each one of the sets $A,S$ and $B$, without loss of generality $I=\{a,s_1,b\}$ where $b\in V(B)$ and $a\in A_2$, then $G - I$ has no perfect matching since $A_1$ is an independent set in $G - I$ of size $q$ but it has $q-1$ neighbours in $G-I$. Hence, in each case, $G$ is equimatchable by Lemma \ref{lem: defn equim}.

\textbf{$G \in \mathcal G_1\cup \mathcal G_2\cup \mathcal G_3 \cup \mathcal  G_4$ is ECE:}
We will show that for every $e=uv \in E(G)$ the graph $G\sm e$ is not equimatchable by considering every possible type of edge. First, suppose $e$ links two vertices in $A$ and $S$, say  without loss of generality $u \in A$ and $v=s_1$. Then for some $b \in B - N(s_1)$ there is a perfect matching $M_b$ in $B-b$, and a perfect matching $M_A$  in $A-u'$ for some $u'$ such that $uu', s_2u'\in E$. 
Now, the set $ M_A \cup M_b \cup \{uv,u's_2\}$ is a matching containing $uv$ and saturating $N(\{u,v\} )$. It follows from Lemma \ref{lem: critical edge iff there is a matching} that $G\sm e$ is not equimatchable. Similarly, if $e$ links two vertices in $B$ and $S$, say  $u \in B$ and $v=s_1$, then letting $M_A$ and $M_u$ being perfect matchings in $A$ and $B-u$ respectively, the set  $M_A \cup M_u \cup \{uv\}$ is a matching containing $uv$ and  saturating $N(\{u,v\})$. So, $G\sm e$ is not equimatchable.  Now, if $e$ belongs to $A$, then there is a perfect matching $M_A$ of $A$ containing the edge $uv$ such that for some $b_i \in N_B(s_i)$ for $i=1,2$,  the set $M_A \cup \{s_1b_1,s_2b_2\}$ is a matching containing $uv$ and  saturating $N(\{u,v\})$. So, $G\sm e$ is not equimatchable. Finally, if  $e$ belongs to $B$, then there exists $b \in B -\{u,v\}$ due to $p\geq 1$, say $bs_i \in E(G)$ for some $i\in \{1,2\}$ such that there exists a perfect matching $M_b$ of $B-b$ containing the edge $uv$, and a vertex $a \in N_A(s_{3-i})$ such that the set $M_b \cup \{s_ib,s_{3-i}a\}$ is a matching containing $uv$ and  saturating $N(\{u,v\})$. So, $G\sm e$ is not equimatchable. Hence $G$ is ECE. 

\textbf{$G \in \mathcal G_5$ is equimatchable:}
By definition, the set $\{a_1,a_2,w\}$ is independent, $s_i$ is complete to $\{b,a_i,w\}$ for $i=1,2$, and $\{a_1,w,a_2\}$ is complete to $A'=A-\{a_1,w,a_2\}$. 
Consider an independent set $I$ of size $3$. If $I$ contains $S$ or $I =\{a_1,a_2,w\}$, then $G - I$ has an odd component implying that $G -I$ has no perfect matching. The only remaining possibility is that $I$ consists of the vertex $b$ and two vertices from $\{a_1,a_2,w\}$. In this case, $\{s_1,s_2\}$ is an independent set in $G \sm I$, but it has a unique neighbour in $G \sm I$. So $G \sm I$ has no perfect matching. Therefore, $G$ is equimatchable by Lemma \ref{lem: defn equim}.

\textbf{$G \in \mathcal G_5$ is ECE:}
Let us show that for every possible $e=uv \in E(G)$, the graph $G\sm e$ is not equimatchable using Lemma \ref{lem: critical edge iff there is a matching}. Let $A'=A-\{a_1,w,a_2\}$. If $e$ links two vertices in $B$ and $S$, say  $u=b$ and $v=s_1$, then for some $a' \in A'$, the set  $\{s_1b,s_2w, a_1a'\}$ is a matching containing $s_1b$ and  saturating $N(\{s_1,b\})$. Similarly, if $e$ links two vertices in $\{a_1,a_2,w\}$ and $S$, say without loss of generality $u \in \{a_1,w\}$ and $v=s_1$, then for $x \in \{a_1,w\}\sm \{u\}$, there is a perfect matching $M_x$ of  the graph induced by $A' \cup \{x\}$ such that the set  $M_x \cup \{s_1u, bs_2\}$  is a matching containing $s_1u$ and saturating $N(\{s_1,u\})$.   Finally, the cases where $e$ links two vertices in $\{a_1,a_2,w\}$ and $A'$,  say without loss of generality $u=a_1$ and $v\in A'$, or two vertices in $A'$ will be handled commonly. In these cases, there is a perfect matching $M_{a_1}$ of  the graph induced by $A' \cup \{a_1\}$ that contains the edge $a_1v$. It follows that the set   $M_{a_1} \cup \{s_1w,s_2a_2\}$ is a matching containing $a_1v$ and  saturating $N(\{a_1,v\})$. Hence $G$ is ECE-graph by Lemma \ref{lem: critical edge iff there is a matching}. 
\end{proof}

Now, we will show that all factor-critical ECE-graphs with connectivity $2$ belong to the family $\F$. To this end, we will first give an equivalent formulation for a graph to belong to the family $\F$ which follows directly from the definitions of Types I, II, III, IV, V forming the family $\F$ (as depicted in Figure \ref{fig:FC- ESE conn 2}).

\begin{prop}\label{prop:member of F}
Let $G$ be a factor-critical graph of order at least $7$ and connectivity $2$.
Then $G$ is a member of $\F$ if and only if there exists a 2-cut $S=\{s_1,s_2\}$ such that $G -S$ has exactly two components $A$ and $B$, and the followings hold:
\begin{itemize}
\item[$(i)$] B is isomorphic to either $K_1$ or $K_{2p+1}$ or $K_{p,p+1}$ for $p\geq 1$. Moreover, for $p\geq 1$, if $B \cong K_{2p+1}$ (resp. $K_{p,p+1}$), then $S$ is partially-complete to $B$ (resp. $(p+1)$-stable set of $B$).
\item[$(ii)$] $S$ is independent set. 
\item[$(iii)$] If $\vert B \vert > 1$, then $A \cup \{s_1,s_2\}$ induces either $K_{2q+2} \sm s_1s_2$ or  $K_{q+1,q+1} \sm s_1s_2$ for $q\geq 1$. If $\vert B \vert = 1$, then we have $N(s_i)= \{b,a_i,w\}$ where $b \in B$, and $a_i,w \in A$ for $i=1,2$, and $A \cong K_{2q} \sm \{a_1a_2,wa_1,wa_2 \}$ for $q\geq 3$.
\end{itemize}
\end{prop}

\setcounter{claim}{0}
\begin{lem}\label{lem: G is ECE then G has five conf}
Let $G$ be a factor-critical graph and connectivity $2$.
If $G$ is $\ECE$-graph, then $G$ is a member of $\F$. 
\end{lem}

\begin{proof}
Suppose that $G$ is a factor-critical ECE-graph with connectivity 2. By Theorem \ref{thm: Favaron connectivity 2 }, there is a 2-vertex-cut $S=\{s_1,s_2\}$ such that $G- S$ has precisely two components $A$ and $B$ as described in items \textit{$(i)$} and \textit{$(ii)$} of Theorem \ref{thm: Favaron connectivity 2 }.  Let $s_1b_1,s_2b_2 \in E(G)$ for vertices $b_1,b_2 \in B$ (where $b_1$ and $b_2$ are distinct if $|B|>1$, and $b_1=b_2$ if $B=K_1$), and let $s_1a_1,s_2a_2 \in E(G)$ for distinct vertices $a_1,a_2 \in A$ (see Figure \ref{fig:FCE graph}). We will prove that $G$ satisfies the conditions $(i),(ii),(iii)$ in Proposition \ref{prop:member of F} to show that  $G$ has one of the five configurations in Figure \ref{fig:FC- ESE conn 2}.

\textit{$(i)$} If $\vert B \vert=1$, then $B=K_1$, and so the claim clearly holds. Thus we may assume $\vert B \vert >1$, and so $|B|\geq 3$. First, note that $A$ is either $K_{2q}$ or $K_{q,q}$  by Theorem \ref{thm: Favaron connectivity 2 }  \textit{$(ii)$}. In addition, we infer that $a_1a_2 \in E(G)$, since otherwise $A$ would be $K_{q,q}$ where $a_1$ and $a_2$ belong to the same $q$-stable set. However, extending $s_1a_1, s_2a_2$ into a maximal matching in $G$ leaves two vertices of $A$ exposed, contradicting to the equimatchability of $G$ by Corollary \ref{cor: def EFC}.

We now claim that for every $b \in N_B(\{s_1,s_2\})$,  $B-b$ is a randomly matchable graph. Without loss of generality, assume $bs_1 \in E(G)$.  Consider a matching $M_a$ containing $bs_1,s_2a_2$ and saturating all vertices but a vertex $a \in A$. This is a minimal matching isolating $a$. Thus, by Lemma \ref{lem: gen isolating matching}, the graph $G\sm (V(M_a)\cup \{a\})$ which is  the graph induced by $B-b$ is randomly matchable.
It then follows from Lemma \ref{lem:randomlymatchable} that $B-b$ is isomorphic to $K_{p,p}$ or $K_{2p}$ for every $b \in N_B(\{s_1,s_2\})$. Therefore, if $B$ is isomorphic to $K_{p,p+1}$ or $K_{p,p+1}+b_1b_2$ as described in  Theorem \ref{thm: Favaron connectivity 2 }, then $N_B(S)$ is included in the $(p+1)$-stable set of $K_{p,p+1}$. Moreover, if $B$ is $K_{2p+1}\setminus b_1b_2$, then $S$ has no neighbour in $B$ other than $b_1$ and $b_2$. 

\begin{claim}\label{cl: B dominated}
If there is a vertex $w \in B\sm N(S)$, then $B-w$ has no perfect matching.
\end{claim}
\begin{cproof}
Assume for a contradiction that there is a vertex $w\in B$ which is not adjacent to $S$ such that $B-w$ has a perfect matching $M$. Clearly, $M$ is a minimal matching isolating $w$, and therefore $A \cup S$ induces a connected randomly matchable graph in $G$ by Lemma \ref{lem: gen isolating matching}. That is, $A \cup S$ induces a graph which is isomorphic to $K_{2q+2}$ or  $K_{q+1,q+1}$ for some $q\geq 1$. It then follows that $s_1s_2 \in E(G)$ since $a_1a_2, a_1s_1, a_2s_2 \in E(G)$. In this case, however, we show that $s_1b_1 \in E(G)$ is not a critical edge. 
Indeed, if $B\cong K_{p,p+1}$ or  $B\cong K_{p,p+1}+b_1b_2$, then $T=N(\{s_1,b_1\})$ contains the $(p)$-stable set of  $K_{p,p+1}$. On the other hand, $T$ contains either $A \cup \{s_2\}$ (which induces a $K_{2q+1}$) or the $(q+1)$-stable set of  $A \cup S$.  We then deduce that every matching of $G$ containing $s_1b_1$ leaves a vertex of $T$ exposed. It can be checked that the same holds if $B\cong K_{2p+1}$ or $B\cong K_{2p+1}\setminus b_1b_2$ (recall that in this case $N_B(\{s_1,s_2\})=\{b_1,b_2\}$). Consequently, there is no matching  containing $s_1b_1$ and saturating $T=N(\{s_1,b_1\})$, contradicting to the criticality of $s_1b_1$.
\end{cproof}\medskip

We have already noticed that if $B$ is $K_{2p+1}\setminus b_1b_2$ then $S$ has no neighbor in $B$ other than $b_1$ and $b_2$. So, Claim \ref{cl: B dominated} implies that $B$ is not isomorphic to $ K_{2p+1} \sm b_1b_2$ for $p\geq 2$ (note that the case $p=1$ corresponds to the graph $K_{1,2}$). We also note that $B$ is not isomorphic to $K_{p,p+1}+b_1b_2$ neither. 
 Indeed, $S$ has no neighbour in $(p)$-stable set of $K_{p,p+1}+b_1b_2$, but for a vertex $b$ in  $(p)$-stable set of $K_{p,p+1}+b_1b_2$, the graph $(K_{p,p+1}+b_1b_2)-b$ has obviously a perfect matching, it contradicts Claim \ref{cl: B dominated}. Hence $B$ is not isomorphic to $K_{p,p+1}+b_1b_2$.
 It then follows from Claim \ref{cl: B dominated} that either $B \cong K_{2p+1}$ and every vertex of $B$ is adjacent to at least one of $s_1,s_2$, or $B \cong K_{p,p+1}$ and every vertex of $(p+1)$-stable set of $B$ is adjacent to at least one of $s_1,s_2$. 

To complete the proof, it remains to show that $s_1$ and $s_2$ has no common neighbour in $B$. To this end, we first enlighten the links between $S$ and $A$ as follows:  

\begin{claim}\label{cl: A cup S}
Both $A\cup \{s_1\}$ and  $A\cup \{s_2\}$ induce either $K_{2q+1}$ or $K_{q,q+1}$.
\end{claim}

\begin{cproof}
We first claim that none of  $s_1$ or $s_2$ is complete to $B\cong K_{2p+1}$ or the $(p+1)$-stable set of $B\cong K_{p,p+1}$. Assume for a contradiction that $s_1$ is complete to $B\cong K_{2p+1}$ or the $(p+1)$-stable set of $B\cong K_{p,p+1}$, then the edge $s_1a_1$ is not critical. Indeed, $N_{G\sm s_1a_1}(\{s_1,a_1\})$ contains all vertices of $B\cong K_{2p+1}$ or the $(p+1)$-stable set of $B\cong K_{p,p+1}$, as well as $N_A(a_1)$ which is either an odd clique $K_{2q-1}$ or the $q$-stable set of $A$. In all cases, there is no matching  containing $s_1a_1$ and saturating $N(\{s_1,a_1\})$, contradicting that $s_1a_1$ is a critical edge in $G$ by Lemma \ref{lem: critical edge iff there is a matching}. 
Thus, for each $i=1,2$, there is a vertex $b \in B$ such that $s_ib \notin E(G)$ where $b$ is a vertex in $B\cong K_{2p+1}$ or the $(p+1)$-stable set of $B\cong K_{p,p+1}$.
This implies that there exists a vertex $b \in B$ with $b\in N(s_2)\sm N(s_1)$ such that  $B-b$ has a perfect matching $P$. In addition, $P\cup \{s_2a_2\}$ is a matching isolating $b$. It then follows from Lemma \ref{lem: gen isolating matching} that $G[A \cup \{s_1\}] \sm \{a_2\}$ is a randomly matchable graph. By symmetry, $G[A \cup \{s_2\}] \sm \{a_1\}$ is randomly matchable as well. Moreover, any vertex  $a\in N_A(s_1)$ (or $a\in N_A(s_2)$) can play the role of  $a_1$, i.e., the graph $G[A \cup \{s_2\}] \sm \{a\}$ is randomly matchable graph for every $a\in N_A(s_1)$. Likewise, the graph $G[A \cup \{s_1\}] \sm \{a\}$ is randomly matchable graph for every $a\in N_A(s_2)$.
This implies the following: if $s_1$ and $s_2$ have a common neighbour in $A$ then each of the sets $A\cup \{s_1\}$ and $A\cup \{s_2\}$ induce cliques of size $2q+1$; if $s_1$ and $s_2$ have no common neighbour in $A$ then each of the sets $A\cup \{s_1\}$ and $A\cup \{s_2\}$ induce $K_{q,q+1}$, and  $s_1a_2, s_2a_1 \notin E(G)$.  
\end{cproof}

Let us now show that $s_1$ and $s_2$ has no common neighbour in $B$. Assume for a contradiction that there is a vertex $b \in B$ such that $s_1b,s_2b \in E$. In this case, we show that the edge $s_1b$ is not critical. Let $R=N_{G\sm s_1b}(\{s_1,b\})$ contains  either $B-b$ (when $B$ is an even clique) or $(p)$-stable set of $B\cong K_{p,p+1}$.   Moreover, $R$ contains either $A \cup \{s_2\}$ (when $A$ is an even clique) or  $q+1$-stable set of  $A\cup \{s_2\}$. In all cases, one can easily check that there is no  matching containing $s_1b$ and  saturating $N(\{s_1,b\})$, contradicting to the criticality of $s_1b$ by Lemma \ref{lem: critical edge iff there is a matching}.\medskip


\textit{$(ii)$}  If $s_1s_2 \in E(G)$, then by the statement (i), $N(\{s_1,s_2\})$ contains $B$ or the $(p+1)$-stable set of $B$ when $B \cong K_{2p+1}$ or $B \cong K_{p,p+1}$, respectively. Thus, there is no matching containing $s_1s_2$ and  saturating $N(\{s_1,s_2\})$, implying that $s_1s_2$ is not a critical edge by Lemma \ref{lem: critical edge iff there is a matching}. It follows that $s_1s_2 \notin E(G)$.\medskip

\textit{$(iii)$} We shall prove this item under two main cases with respect to the size of $B$.\medskip

\textbf{\textit{Case 1:}} $\vert B \vert > 1$.

As we have already shown in Claim \ref{cl: A cup S} that   $G[A \cup \{s_1,s_2\}]$ is either $K_{q+1,q+1} \sm s_1s_2$ or  $K_{2q+2} \sm s_1s_2$.\medskip

\textbf{\textit{Case 2:}} $\vert B \vert = 1$.

Let $b$ be the unique vertex in $B$. By Theorem \ref{thm: Favaron connectivity 2 }-(ii), $A'=A-\{a_1,a_2\}$ is either $K_{q-1,q-1}$ or $K_{2q-2}$. Recall that $q\geq 2$ by Remark \ref{rem:7 vertices}. We first note that if $s_1$ or $s_2$ has only two neighbours in $G$, say $d_G(s_1)=2$, then it burns down to the Case 1 as we take the 2-cut $S=\{a_1,s_2\}$. Thus, $d_G(s_i)\geq 3$ for $i=1,2$.
Similarly, if none of $s_1,s_2$ has a neighbour in $A'$, then it boils down to the Case 1 as we take the 2-cut $S=\{a_1,a_2\}$. Thus there exists $w\in N_{A'}(S)$. We then deduce that if $s_1$ or $s_2$ has no neighbour in $A'$, say this is $s_1$, then  $a_2$ would be a common neighbour of $s_1$ and $s_2$  since $d(s_1)\geq 3$ and $s_1s_2 \notin E$ by the item $(ii)$. It follows that for every pair $x,y\in \{a_1,a_2,w\}$ there exists a matching which saturates $\{s_1,s_2,x,y\}$ and isolates $b$; thus $A-\{x,y\}$  is randomly matchable by Lemma \ref{lem: gen isolating matching}. 
 Hence we may exchange the roles of  $a_2$ and $w$ as we desired. That is, if we define  $A'=A-\{a_1,w\}$, then both $s_1$ and $s_2$ would have a neighbour in $A'$.  
 It follows that $s_1$ and $s_2$ has a common neighbor in $A'$ (and $s_1$ has no other neighbor in $A'$). The only remaining case is if both $s_1$ and $s_2$ have distinct neighbors in $A'$. 
 
Combining the two cases, in what follows, we assume that each of $\{s_1,s_2\}$ has at least one neighbour in $A'$. So there exist $w_1,w_2 \in A'$ such that $w_1 \in N(s_1)$ and $w_2 \in N(s_2)$  (with possibly $w_1= w_2$). Since $\{s_1w_1,s_2a_2\}$ is a matching isolating $b$, the graph $G[A' \cup \{a_1\}] - \{w_1\}$ is randomly matchable by Lemma \ref{lem: gen isolating matching}. Similarly, $G[A' \cup \{a_2\}] - \{w_2\}$  is randomly matchable since $\{s_1a_1,s_2w_2\}$ is a matching isolating $b$.
It then follows that if $A'$ is isomorphic to $K_{q-1,q-1}$ with a bipartition $R$ and $T$, then  for each $i=1,2$, we deduce that both $a_i$ and $w_i$ are complete to either $R$ or $T$. That is, all neighbours of $s_i$ in $A$ are complete to either $R$ or $T$. Similarly, if  $A'$ is isomorphic to  $K_{2q-2}$, then  for each $i=1,2$, we deduce that $a_i$ is complete to $A'-w_i$.

%
%
%
%

In the remaining of the proof, we distinguish all possible cases according to $A'=K_{q-1,q-1}$ and $A'=K_{2q-2}$ and vertices $w_1,w_2,s_1,s_2$. We will see that there is no ECE-graph where  $A'=K_{q-1,q-1}$, and the only possible configuration of an ECE-graph with $A'=K_{2q-2}$ corresponds to Type V (see Figure \ref{fig: FC- ESE graph (e)}).  \medskip

\textbf{\textit{ $\quad$ Subcase 2.1:}} Suppose that $A' \cong K_{q-1,q-1}$, and there exists $w_i \in N(s_i)$ for each $i=1,2$ such that  $w_1 \neq w_2$ and $w_1w_2 \in E(G)$.

We claim that $s_1a_2, s_2a_1  \notin E(G)$. Assume for a contradiction that  without loss of generality $s_1a_2 \in E(G)$. Then $\{s_1a_2,s_2w_2\}$ is a matching isolating $b$. By Lemma \ref{lem: gen isolating matching},  the remaining graph $H=G[A' \cup \{a_1\}] - \{w_2\}$ must be randomly matchable. We then say that $a_1$ and $w_2$ are complete to the same $(q-1)$-stable set of $A'$. Recall also that  $a_1,w_1 \in V(H)$ are complete to the same $(q-1)$-stable set of $A'$, and $w_1w_2\in E(G)$. This implies that $H$ is not randomly matchable whenever $q>2$. If however $q=2$, i.e, $A'=\{w_1,w_2\}$, then $H$ being randomly matchable implies that $a_1w_1 \in E(G)$. Now, noting that $a_1w_2 \in E(G)$, the edge $a_1s_1$ is not critical in $G$ since all vertices $w_1,w_2,a_2,b$ belong to  $N(\{a_1,s_1\}) \sm \{a_1,s_1\}$ whereas there is no matching in $G \sm \{a_1,s_1\}$ saturating $\{w_1,w_2,a_2,b\}$, a contradiction by Lemma \ref{lem: critical edge iff there is a matching}. Thus, we conclude that $s_1a_2 \notin E(G)$. By symmetry, we also have $s_2a_1 \notin E(G)$. 
Moreover, $a_1$ is adjacent to $a_2$, since otherwise for a perfect matching $M$ in $A'-\{w_1,w_2\}$, the set $M \cup \{s_1w_1,s_2w_2\} $ is a  perfect matching in $G \sm \{a_1,a_2,b\}$ where $ \{a_1,a_2,b\}$ is an independent set, it is contradiction to the equimatchability of $G$ by Lemma \ref{lem: defn equim}. 

Now, assume $A'$ has a third vertex $w_2'\in N(s_2)$  with $w_2'\notin\{w_1, w_2\}$, then there is also $w_1'$ in $A$ since  $A' \cong K_{q-1,q-1}$. Then, $a_1$ is adjacent to both $w_2$ and $w_2'$, similarly $a_2$ is adjacent to both $w_1$ and $w_1'$. Now we claim that the vertex $a_2$ has no neighbor in the stable set $R$ of $A' $ containing $w_2$ and $w_2'$. Indeed, if $a_2w\in E(G)$ for some $w\in R$ (which is not necessarily a neighbor of $s_2$) then letting without loss of generality $w'_2$ to be a neighbor of $s_2$ different from $w$, the matching $\{s_1a_1,s_2w_2'\}$ isolates $b$. However, the remaining graph $G-\{b,s_1,s_2,a_1,w_2'\}$ is not randomly matchable since the vertices $w_1,w,a_2,w'_1$ induce $K_4-w_1w_1'$ in $G$, a contradiction to the equimatchability of $G$ by Lemma \ref{lem: gen isolating matching}. 
Thus, $R\cup \{a_2\}$ is a stable set in $G$. In this case however, the edge $s_1a_1$ is not critical because there is no matching containing $s_1a_1$ and  saturating $N=N(\{s_1,a_1\})$; indeed $N$ contains the independent set $R \cup \{b,a_2\}$ of size $q+1$ while $G\sm \{s_1,a_1\}$ has $2q+1$ vertices. We therefore conclude that for $i=1,2$, each $s_i$ is adjacent to only $w_i$ in $A'$. 


We next claim that $w_1a_1,w_2a_2 \in E(G)$. Assume it is not true and let without loss of generality $w_2a_2 \notin E(G)$. In this case, the edge $s_1a_1$ is not critical. To show this, let us first note that $a_2$ has no neighbor in the $(q-1)$-stable set $R$ of $A'$ containing $w_2$ since $\{s_1w_1,s_2w_2\}$ is a matching isolating $b$, the graph $G[A - \{w_1,w_2\}]$ is randomly matchable by Lemma \ref{lem: gen isolating matching}. Now, $N=N(\{s_1,a_1\})$ contains the stable set $R\cup\{b,a_2\}$ of size $(q+1)$ while $G \sm \{a_1,s_1\}$ has $2q+1$ vertices. So, there is no matching containing $s_1a_1$ and  saturating $N$, a contradiction by Lemma \ref{lem: critical edge iff there is a matching}.

Lastly, we show that the obtained graph is not ECE as follows. The matching $\{w_1a_1,bs_2\}$ is a matching isolating $s_1$, however, the remaining graph $H=G-\{w_1,a_1,b,s_2,s_1 \}$ is isomorphic to $K_{q-2,q}+a_2w_2$. Clearly, $H$ is not randomly matchable if $q>2$  which yields a contradiction to the equimatchability of $G$ by Lemma \ref{lem: gen isolating matching}. In addition, if $q=2$, then $a_1s_1$ is not a critical edge, since $N(\{a_1,s_1\})\sm \{a_1,s_1\}$ consists of the vertices $w_1,w_2,a_2,b$ which cannot be saturated by any matching in $G \sm a_1s_1$. We conclude that there is no such type of ECE-graph. \medskip

\textbf{\textit{ $\quad$ Subcase 2.2:}} Suppose that $A' \cong K_{q-1,q-1}$, and there exists $w_i \in N(s_i)$ for each $i=1,2$ such that  $w_1 \neq w_2$ and $w_1w_2 \notin E(G)$.

Consider the graph $A' \cong K_{q-1,q-1}$ with a bipartition $R$ and $T$, and let without loss of generality $w_1,w_2 \in R$ and take $w_1',w_2' \in T$.  If $a_1a_2 \in E(G)$, then for a perfect matching $M$ in $A'-\{w_1,w_1',w_2,w_2'\}$,   the set $M\cup \{s_1w_1,s_2w_2,a_1a_2\}$ is a perfect matching in $G \sm \{w_1',w_2',b\}$ where $ \{w_1',w_2',b\}$ is a stable set,  it is a contradiction to the equimatchability of $G$ by Lemma \ref{lem: defn equim}. Thus $a_1a_2 \notin E(G)$. Besides, we observe that if none of the edges $w_1a_1$ and $w_2a_2$ is present, then $G$ is bipartite with bipartition $T\cup \{s_1,s_2\}$ and $R \cup \{a_1,a_2,b\}$. However, bipartite graphs are not factor-critical, a contradiction. So we conclude  that one of $w_1a_1,w_2a_2$ appears in $G$. If both of them are present, then for a perfect matching $M$ in $A'-\{w_1,w_1',w_2,w_2'\}$, the set  $M\cup \{w_1a_1,w_2a_2,bs_2\}$ is a perfect matching in $G \sm \{w_1',w_2',s_1\}$ where $\{w_1',w_2',s_1\}$ is a stable set, contradicting that $G$ is equimatchable by Lemma \ref{lem: defn equim}. Therefore precisely one of $w_1a_1$ or $w_2a_2$ appears on the graph $G$, without loss of generality assume $w_1a_1 \notin E(G)$ and $w_2a_2 \in E(G)$. Then, we claim that $a_1s_2 \notin E(G)$, since otherwise $\{a_1s_2,w_1s_1\}$ is a matching isolating $b$ whereas the remain graph $G-\{a_1,s_2,w_1,s_1,b\}$ is not randomly matchable since it contains the edge $w_2a_2$, a contradiction by  Lemma \ref{lem: gen isolating matching}.
Likewise, if $s_2$ is adjacent to $w_1$, then $\{a_1s_1,w_1s_2\}$ is a matching isolating $b$ whereas the remain graph $G-\{a_1,s_1,w_1,s_2,b\}$ is not randomly matchable,  a contradiction by  Lemma \ref{lem: gen isolating matching}. Therefore $\{w_1,a_1,s_2\}$ is a stable set. Then, for a perfect matching $M$ in $A'-\{w_1,w_1'\}$, the set $M\cup \{w_1'a_2,bs_1\}$ is a perfect matching in $G \sm \{w_1,a_1,s_2\}$, contradicting that $G$ is equimatchable by Lemma \ref{lem: defn equim}. Consequently, there is no such type of ECE-graph. \medskip

\textbf{\textit{ $\quad$ Subcase 2.3:}} Suppose that $A' \cong K_{q-1,q-1}$, and the vertices $s_1,s_2$ have a unique neighbour $w_1 = w_2=w$ in $A'$.

Consider the graph $A' \cong K_{q-1,q-1}$ with a bipartition $R$ and $T$, and let  $w \in R$. Then, for $i=1,2$ each $a_i$ is adjacent to all vertices in $T$. Besides, if $w$ is adjacent to none of $a_1$ and $a_2$, then $G$ induces a bipartite graph with a bipartition $T\cup \{s_1,s_2\}$ and $R \cup \{a_1,a_2,b\}$. However, bipartite graphs are not factor-critical, a contradiction. So there exists at least one edge between $w$ and $\{a_1,a_2\}$, without loss of generality assume $wa_1 \in E(G)$. Then, there is no matching containing $a_1w$ and  saturating $N=N(\{a_1,w\})$ since $N$ includes $T \cup \{s_1,s_2\}$ which is an independent set of size $q+1$ whereas $G \sm \{a_1,w\}$ has $2q+1$ vertices, a contradiction to the fact that the edge $wa_1$ is critical by Lemma \ref{lem: critical edge iff there is a matching}. Hence, $s_1$ and $s_2$ has no common neighbour in $A'$. Consequently, there is no such type of ECE-graph. \medskip

\textbf{\textit{ $\quad$ Subcase 2.4:}} Let $A' \cong K_{2q-2}$ for $q\geq 2$.

Recall first that for each $i=1,2$, $a_i$ is complete to $A'-w_i$.  Let us first show that $a_1a_2 \notin E(G)$. Indeed, if $a_1a_2 \in E(G)$, then $N(\{s_1,a_1\})$ contains all vertices in $V(G)\sm \{s_2\}$. It can be observed that  there is no matching containing $s_1a_1$ and  saturating $N(\{s_1,a_1\})$,
a contradiction to the criticality of the edge $s_1a_1$ by Lemma \ref{lem: critical edge iff there is a matching}.  We next claim that each vertex of $s_1,s_2$ has a unique neighbour in $A'$, and they are the same. Indeed, if $w_1\neq w_2$, then $\{s_1w_1,s_2w_2\}$ would be a matching isolating $b$. However, $A \sm \{w_1,w_2\}$ is not randomly matchable since $a_1a_2 \notin E(G)$, a contradiction by Lemma \ref{lem: gen isolating matching}. Hence the vertices $s_1,s_2$ have a unique neighbour in $A'$, say $w=w_1=w_2$. If one of the edges $a_1w,a_2w$ is present in $G$, say $a_1w\in E(G)$, then this edge is not critical since there is no matching containing $a_1w$ and  saturating $N(\{a_1,w\})$, 
a contradiction by Lemma \ref{lem: critical edge iff there is a matching}. Therefore $\{a_1,a_2,w\}$ is a stable set. Lastly, observe that if $A'-w$ consists of a single vertex, then $G$ induces a bipartite graph. However, bipartite graphs are not factor-critical. Hence  $q \geq 3$  and  $A'-w \cong K_{2q-3}$ (see Figure \ref{fig: FC- ESE graph (e)}). 
\end{proof}

By combining Lemmas \ref{lem:All graphs belonging to F are ECE} and \ref{lem: G is ECE then G has five conf}, and noting that factor-critical $\ECE$-graphs have at least 7 vertices by Remark \ref{rem:7 vertices}, we obtain the characterization of factor-critical ECE-graphs with connectivity $2$ given in Theorem \ref{thm:chrc of ECE conn 2}. 

\section{An overview of subclasses of equimatchable graphs}\label{sec:compare}
Let us give a comparison of $\VCE$-graphs and $\ECE$-graphs with other subclasses of equi-matchable graphs that are well-studied in the literature, namely factor-critical equimatchable  (EFC) graphs and edge-stable equimatchable (ESE) graphs (see Figure \ref{fig: ECE VCE}). To this end, we define the following disjoint families of graphs;

\begin{itemize}
\item[-] $\mathcal{A}$ is the class of EFC-graphs admitting a vertex $v$ such that $G- v$ is isomorphic to $K_{2r}$  for some integer $r\geq 2$ and  $2\leq d(v) \leq 2r-2$. 
\item[-] $\mathcal{B}$ is the class of EFC-graphs admitting a vertex $v$ such that $G- v$ is isomorphic to $K_{r,r}$ for some integer $r\geq 2$ where $v$ is adjacent to at least two adjacent vertices in $K_{r,r}$, and $v$ has at least one non-neighbor in each one of the $(r)$-stable sets of $K_{r,r}$. 
\item[-] $\mathcal{C}$ is the class of graphs  which are isomorphic to $K_{3}$, $K_{2r+1}$ or $K_{2r+1}\setminus e$ for an edge $e\in E(K_{2r+1})$  for some integer $r\geq 2$. 
\item[-] $\mathcal{D}$ is the class of EFC-graphs admitting a vertex $v$ such that $G- v$ is isomorphic to $K_{r,r}$  for some integer $r\geq 2$ where $v$ is adjacent to at least two adjacent vertices in $K_{r,r}$, and $v$ is complete to an $(r)$-stable set of $K_{r,r}$. 
\item[-] $\mathcal{E}$ is the class of EFC-graphs with a cut vertex. 
\end{itemize}

Note that an EFC-graph must be connected since it is factor-critical. Thus, EFC-graphs consists of two parts; EFC-graphs with a cut vertex and 2-connected EFC-graphs. We observe that all graphs in $\mathcal{A} \cup \mathcal{B} \cup \mathcal{C} \cup \mathcal{D}$ are 2-connected.

We first investigate the relation of the graph class ESE with the above defined graph classes. It has been shown in \cite{ESE} that $\ESE$-graphs are either 2-connected factor-critical or bipartite.

\begin{prop}\label{prop:ESE-ABCDE}
We have $(\mathcal{A}\cup \mathcal{B} \cup \mathcal{E})\cap \ESE = \emptyset$ and $(\mathcal{C}\cup \mathcal{D})\subseteq \ESE$.
\end{prop}
\begin{proof}
Let us first show that no graph in $\mathcal{A}\cup \mathcal{B}$ is ESE.
Recall that an equimatchable graph $G$ is edge-stable if $G \setminus e$ is also equimatchable for any $e \in E(G)$.
In \cite{ESE}, it has been proved that an EFC-graph is edge-stable if and only if there is no induced $\overline{P_3}$ in $G$ such that $G \setminus \overline{P_3}$ has a perfect matching where $\overline{P_3}$ is the complement of a path on 3 vertices. Consider a graph $G\in \mathcal{A}\cup \mathcal{B}$ with the vertex $v$ as described in the definitions of the families $\mathcal{A}$ and $\mathcal{B}$.  Then there exist two adjacent non-neighbors $x$ and $y$ of $v$ such that   $\{v,x,y\}$ induces a $\overline{P_3}$ in $G$, and $G\setminus \overline{P_3}$ has a perfect matching (since $G\setminus \overline{P_3}$ is isomorphic  to $K_{2r-2}$ or $K_{r-1,r-1}$). Therefore, $(\mathcal{A}\cup \mathcal{B})\cap \ESE =\emptyset$. Besides, ESE-graphs with a cut-vertex are bipartite \cite{ESE}, thus not factor-critical. It follows that $\mathcal{E}\cap \ESE = \emptyset$.

On the other hand, for $r\geq 2$, a clique $K_{2r+1}$ and $K_{2r+1}\setminus xy$ are ESE by Theorem 12 in \cite{ESE} as well as $K_3$ which is an ESE-graph. So every graph in $\mathcal{C}$ is an ESE-graph. Besides, if $G$ is a graph in $\mathcal{D}$ with the vertex $v$ as stated, the graph $G$ has no $\overline{P_3}$ as an induced subgraph since $v$ is complete to an $(r)$-stable set of $G-v$. Thus, every graph in $\mathcal{D}$ is an ESE-graph. It follows that $(\mathcal{C}\cup \mathcal{D})\subseteq \ESE$. 
\end{proof}

Let us now establish the links between $\VCE$-graphs, $\ECE$-graphs and the families $\mathcal{A}, \mathcal{B}, \mathcal{C}, \mathcal{D}$ and $\mathcal{E}$. 
First of all, it is known that both VCE-graphs and ECE-graphs are 2-connected (by Proposition \ref{prop:VCE-graphs are 2-connected} and Lemma \ref{lem:ECE 2 connected} respectively); thus $\VCE\cap \mathcal{E}=\emptyset$ and $\ECE\cap \mathcal{E}=\emptyset$. Moreover, we have clearly $\ESE \cap \ECE=\emptyset $ by definition of these classes.

We next show that the graphs in $\mathcal{A}\cup \mathcal{B} \cup \mathcal{C}\cup \mathcal{D}$ are not $\ECE$.

\begin{prop}\label{prop:K2r-not-ece}
Let $G$ be a connected graph with $2r+1$ vertices for some $r\geq 1$. If $G$ contains one of $K_{2r},K_{r,r}$ as an induced subgraph, then $G$  is not $\ECE$. 
\end{prop}
\begin{proof}
Let $G$ be a connected graph with $2r+1$ vertices, and  suppose that $G$ contains one of $K_{2r},K_{r,r}$ as an induced subgraph. Then there exists a vertex $v\in V(G)$ such that $G-v$ is isomorphic to $K_{2r}$ or $K_{r,r}$. Since $G$ is connected, the vertex $v$ is adjacent to a vertex $u$ in $V(G-v)$. Note that $G\setminus uv$ contains one of $K_{2r}$ and $K_{r,r}$ as an induced subgraph. So every maximal matching in $G\setminus uv$ is of size $r$. This means that the edge $uv$ is not critical, thus $G$ is not $\ECE$.
\end{proof}

 By Proposition \ref{prop:K2r-not-ece} and the definitions of the families $\mathcal{A}, \mathcal{B}, \mathcal{C}, \mathcal{D}$, we have the following:
 
 \begin{cor} \label{cor:ECEint}
 $(\mathcal{A}\cup \mathcal{B} \cup \mathcal{C}\cup \mathcal{D})\cap \ECE = \emptyset$.
 \end{cor}
 
Since VCE-graphs are equivalent to 2-connected $(K_{2r},K_{r,r})$-free EFC-graphs on $2r+1$ vertices (by Theorem \ref{thm:main-vce}), and ECE-graphs on $2r+1$ vertices do not contain $K_{2r}$ or $K_{r,r}$ (by Proposition \ref{prop:K2r-not-ece}), we have the following:

\begin{cor}\label{cor:ece-vce}
All factor-critical $\ECE$-graphs are $\VCE$.  
\end{cor}

As we have $(\mathcal{A}\cup \mathcal{B} \cup \mathcal{C}\cup \mathcal{D}) \cap \ECE =  \emptyset$ by Corollary \ref{cor:ECEint}, Corollary \ref{cor:ece-vce} implies that $\mathcal{A}\cup \mathcal{B} \cup \mathcal{C}\cup \mathcal{D}\cup \mathcal{E} \subseteq \EFC\sm \VCE $. It remains to show that VCE is equivalent to the class $\EFC\sm (\mathcal{A}\cup \mathcal{B} \cup \mathcal{C}\cup \mathcal{D}\cup \mathcal{E})$.

\begin{cor}
$\EFC\sm \VCE= \mathcal{A}\cup \mathcal{B}\cup \mathcal{C}\cup \mathcal{D}\cup \mathcal{E}.$
\end{cor}
\begin{proof}
Let $G$ be a graph in $\EFC\sm \VCE$. Then, by Proposition \ref{prop:EFC-notVCE}, $G$ has a vertex $v$ such that  every component of $G- v$ is isomorphic to $K_{r,r}$ or $K_{2t}$ for some $r, t\geq 1$   and where $v$ is adjacent to at least two adjacent vertices of $G-v$. If $v$ is a cut-vertex then $G\in \mathcal{E}$. Assume $v$ is not a cut-vertex, then $G-v$ is a connected graph on $2r$ vertices. If $G-v$ is a $K_{r,r}$ for $r\geq 2$ then $G\in \mathcal{B}\cup \mathcal{D}$. If however $G-v$ is a $K_{2r}$  then $G\in \mathcal{A}\cup \mathcal{C}$. It follows that the families $\mathcal{A}, \mathcal{B}, \mathcal{C}, \mathcal{D}$ form together all 2-connected graphs in $\EFC \sm \VCE$ as decribed in Proposition \ref{prop:EFC-notVCE}. 
\end{proof}

 As we stated above, EFC-graphs with a cut vertex are equivalent to the family  $\mathcal{E}$ while 2-connected EFC-graphs consist of three disjoint subclasses;  $\mathcal{A}\cup \mathcal{B}$, $\mathcal{C}\cup \mathcal{D}$ and VCE-graphs. Let us also recall that ECE-graphs are either factor-critical, or bipartite, or even cliques by Theorem \ref{thm:ece-categories}. This completes the full containment relationship between VCE-graphs, ECE-graphs, ESE-graphs and EFC-graphs as illustrated in Figure \ref{fig: ECE VCE} where these classes are represented by sets VCE, ECE, ESE and EFC respectively, and FC means factor-critical.

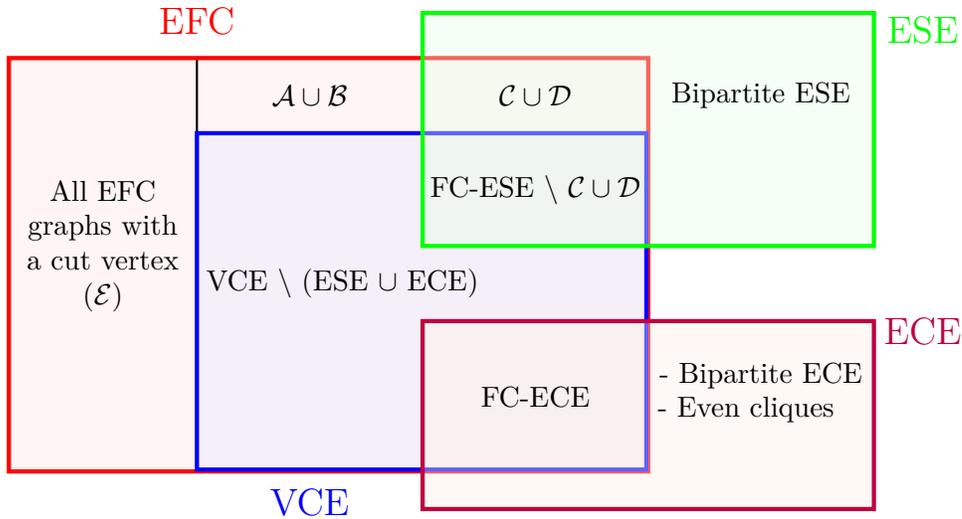
\begin{figure}[htb]
\centering
\begin{tikzpicture}[thick,draw opacity=1,scale=.5,every text node part/.style={align=center}]
\begin{scope}[fill opacity=.4]
\draw[red,ultra thick,fill=red!10!white] (1,0) rectangle (18,11);
\draw[blue,ultra thick,fill=blue!10!white] (6,0.05) rectangle (17.95,9);
\draw[green,ultra thick,fill=green!10!white] (12,6) rectangle (24,12.2);
\draw[purple,ultra thick,fill=orange!10!white] (12,-1) rectangle (24,4);
\end{scope}
\begin{scope}
\draw (6,9.05)--(6,10.95);
\node at (3.5,6)    { All EFC \\ graphs with \\ a cut  vertex \\  ($\mathcal{E}$)};
\node at (6,12)    {\color{red} {\Large EFC}};
\node at (15,2)    { FC-ECE};
\node at (9,-0.8)    {\color{blue} \Large VCE};
\node at (25.3,3.75)    {\color{purple} \Large ECE};
\node at (25.3,11.75)    {\color{green} \Large ESE};
\node at (10,5)    { VCE $\setminus$ (ESE $\cup$ ECE) };
\node at (21,2.1)    { - Bipartite ECE \\ - Even cliques  \hspace*{2.7mm}  };
\node at (15,7.5)    { FC-ESE $\setminus$ $\mathcal{C}\cup \mathcal{D}$};
\node at (9,10)    { $\mathcal{A}\cup \mathcal{B}$};
\node at (15,10)    { $\mathcal{C}\cup \mathcal{D}$};
\node at (21,10)    { Bipartite ESE};
\end{scope}
\end{tikzpicture}
\caption{The world of equimatchable graphs.}
\label{fig: ECE VCE}
\end{figure}

\section{Conclusion}\label{sec:conclusion}

In this paper, we shed light on the structure of equimatchable graphs from a new perspective, namely the criticality with respect to vertex removals and edge removals. We first showed that $\VCE$-graphs boil down to factor-critical equimatchable graphs apart from a few simple exceptions. We also note that factor-critical $\ECE$-graphs form a subclass of $\VCE$-graphs.  This motivates our studies on factor-critical $\ECE$-graphs, whose structure can be analyzed according to their connectivity \cite{Favaron, Kotbic}. We give a full characterization of factor-critical $\ECE$-graphs with connectivity 2.

It remains to characterize factor-critical $\ECE$-graphs with connectivity at least 3. We also investigated this case and obtained a characterization for most of the situations, leaving open a few cases. However, we prefer not reporting these results in this paper due to several reasons. First, the proofs consist of long and technical case analyses. Second, they leave some open cases, thus not providing a full characterization. Last but not least, all these results are based on the results of Eiben and Kotrbcik \cite{Kotbic} about the connectivity of factor-critical equimatchable graphs which is an unpublished manuscript. Nevertheless, let us give a quick overview of our findings which might shed light on forthcoming studies in this direction. The reader is referred to the Appendix. 

Let $G$ be a factor-critical $\ECE$-graph with connectivity $k\geq 3$. Then, if $G$ is a graph with $|G|> 2k+1$ and minimum degree greater than $k$ then we have a full characterization as follows. If $k=3$ then $G$ belongs to three possible categories that can be described in a similar way to the types forming the family $\F$ (for connectivity 2). If $k\geq4$, then  $\overline{G}$ is a maximal triangle-free graph. On the other hand,  the case where $|G|> 2k+1$ and the minimum degree is $k$ is open. For relatively small graphs, that is $|G|\leq 2k+1$, we know that the connectivity is at least 4 but their characterization is again open. It would be interesting to obtain a full characterization of factor-critical $\ECE$-graph with connectivity $k\geq 3$, which may require the development of stronger tools to embody various cases.

%
%
%
%
%
%
%
%

\bibliographystyle{plain}

\newpage
\section*{Appendix}

Here, we provide a partial characterization of factor-critical $\ECE$-graphs with connectivity $k\geq 3$. This section consist of two parts; factor-critical ECE-graphs with connectivity 3, and 4-connected factor-critical ECE-graphs.  All results in this section are based on the results of Eiben and Kotrbcik  \cite{Kotbic}. One of the key results is the following:

\begin{nlem} \cite{Kotbic} \label{lem:kotbic two-comp}
Let $G$ be a $k$-connected EFC-graph with a $k$-cut set $S$, where $k\geq 2$. If $G$ has at least $2k+3$ vertices, then $G - S$ has precisely two components.
\end{nlem}

We show in Proposition \ref{prop. G-S has two comp} that for a $k$-connected factor-critical ECE-graph $G$ with $k\geq 3$, the graph  $G - S$ has precisely two components for a $k$-cut set $S$ (without any condition on the order of $G$ as in Lemma \ref{lem:kotbic two-comp}). In Lemmas \ref{lem: FC ECE,  connectivity 3} and \ref{lem:component-2-vertices},  we first characterize factor-critical ECE-graphs with connectivity 3 where each one of these components has at least two vertices; we figure out three types of graphs in this class (see Figure \ref{fig:ECE, conn 3}). The case where $G-S$ has a component with a single vertex remains open. 

Later we investigate 4-connected factor-critical ECE-graphs. We show in Lemma \ref{lem:conn 4- three equiv} that if $G$ is such a graph with  $2k+3$ vertices for $k\geq 4$ and $\delta(G)>k$, where $\delta(G)$ is the minimum degree of a vertex in $G$, then $\overline{G}$ is a maximal triangle-free graph. 
The case $\delta(G)=k$ remains open as well as the characterization of those 4-connected factor-critical ECE-graphs having at most $2k+1$ vertices.

\vspace*{1em}
\subsection*{A1. Factor-critical ECE-graphs with connectivity 3} ~~\medskip

We now deal with the edge-criticality of 3-connected equimatchable factor-critical graphs.

\begin{nthm} \cite{Kotbic} \label{thm:kotbic conn 3}
Let $G$ be a $k$-connected equimatchable factor-critical graph with at least $2k + 3$ vertices and a $k$-cut $S$ such that $G - S$ has two components with at least 3 vertices, where $k\geq 3$. Then $G \sm S$ has exactly two components and both are complete graphs.
\end{nthm}

We have detected all ECE-graphs having 7 vertices by using a computer program written in Python-Sage. There are only $4$ such graphs shown in Figure \ref{fig:7-vertices}. 

\begin{figure}[htb]
\centering     
\includegraphics[scale=.5]{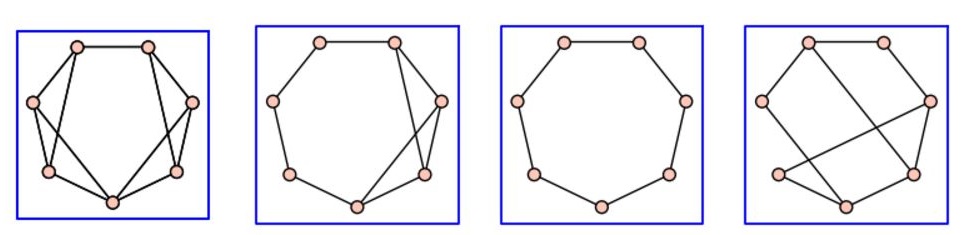}
\caption{All factor-critical ECE-graphs with 7 vertices.}
\label{fig:7-vertices}
\end{figure}

Since all factor-critical ECE-graphs of order at most 7 have a 2-cut set (see Figure \ref{fig:7-vertices}), we say that any 3-connected factor-critical ECE-graph has at least 9 vertices. So, we have the following by Lemma \ref{lem:kotbic two-comp}.

\begin{nprop}\label{prop. G-S has two comp}
If $G$ is a $k$-connected factor-critical $\ECE$-graph with $k$-cut $S$ for $k\geq 3$, then $G - S$ has precisely two components.
\end{nprop}

Let $G$ be a factor-critical graph with connectivity $3$. Suppose that for a 3-cut $S=\{s_1,s_2,s_3\}$, the graph $G -S$ has two components $A$ and $B$ with at least $3$ vertices. 
Then we introduce three possible configurations for the graph $G$ with respect to $A,B$ and $S$ as follows.

\begin{itemize}
\item \textbf{Type VI}: $A\cong K_{2p+1}$, $B \cong K_{2q+1}$ for $p,q\geq 1$ and there exist $a\in A$ and $b\in B$ such that $s_1$ is complete to both $A-a$ and $B-b$, $s_2$ is complete to  $B\cup \{a\}$, and $s_3$ is complete to $A\cup \{b\}$  (see Figure \ref{fig: FC- ESE graph (aa)}). \medskip	
\item \textbf{Type VII}: $A\cong K_{2p+1}$, $B \cong K_{2q+1}$ for $p,q\geq 1$ and there is a non-empty partition $A_1,A_2,A_3$  of $A$  such that each $s_i$ is complete to both $B$ and $A-A_i$ (see Figure \ref{fig: FC- ESE graph (bb)}). \medskip
\item \textbf{Type VIII}: $A\cong K_{2p+1}$ (resp. $K_{2p}$), $B \cong K_{2q+1}$ (resp. $K_{2q}$)  for $p,q\geq 1$ (resp. $p,q\geq 2$) and there is a non-empty partition $A_1,A_2$  of $A$  such that $s_1s_2\in E(G)$,  each one of $s_1,s_2$ is complete to both $B$ and $A_1$, and $s_3$ is complete to both $B$ and $A_2$ (see Figure \ref{fig: FC- ESE graph (cc)}). \medskip
\end{itemize}

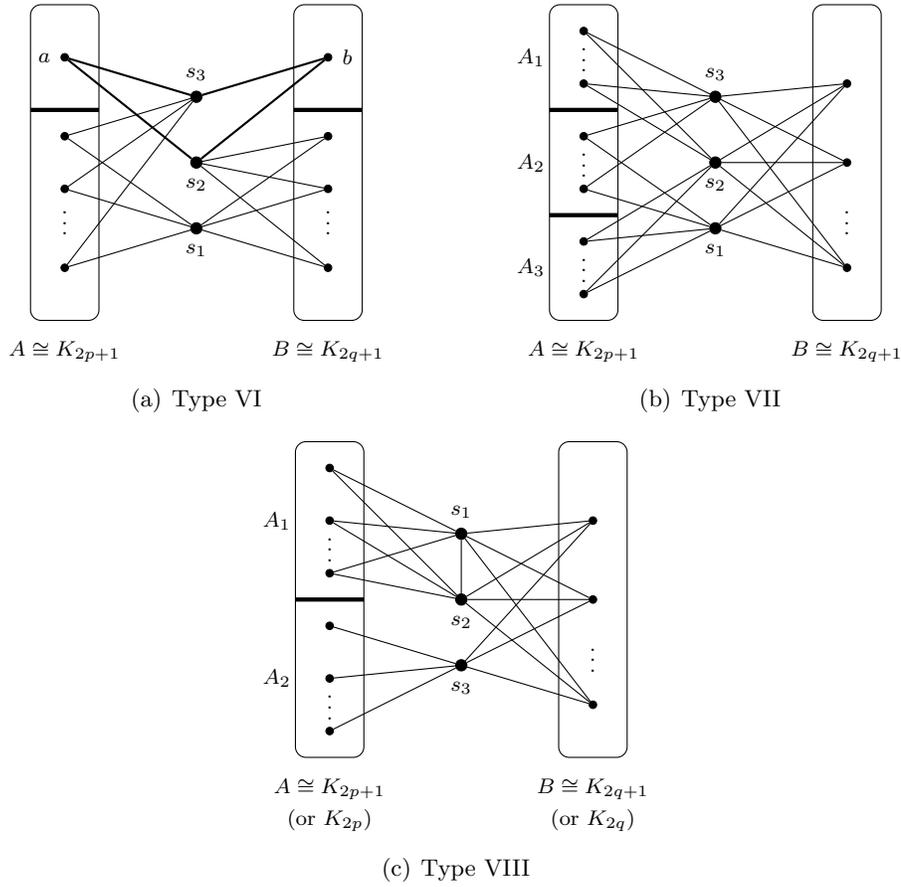
\begin{figure}[htb]
\centering     
\subfigure[Type VI]{\label{fig: FC- ESE graph (aa)}
\begin{tikzpicture}[scale=.7]
\draw [rounded corners] (-.65,-3) rectangle (.65,3);
\draw[ultra thick] (-.65,1)--(.65,1);
\node [noddee2] at (0,2) (a) [label=left: \scriptsize $a$]  {};
\node [noddee2] at (0,.5) (a1)  {};
\node [noddee2] at (0,-.5) (a2)  {};
\node  at (0,-1)   {\scriptsize $\vdots$};
\node [noddee2] at (0,-2) (a3)  {};
\draw [rounded corners] (4.35,-3) rectangle (5.65,3);
\draw[ultra thick] (4.35,1)--(5.65,1);	
\node [noddee2] at (5,2) (b) [label=right: \scriptsize $b$]  {};	
\node [noddee2] at (5,.5) (b1)  {};
\node [noddee2] at (5,-.5) (b2)  {};
\node  at (5,-1)  {\scriptsize $\vdots$};
\node [noddee2] at (5,-2) (b3)  {};
\node [noddee1] at (2.5,1.25) (s1) [label=above: \scriptsize $s_3$] {}
	edge [thick] (a)
	edge [thick] (b)
	edge [] (a1)
	edge [] (a2)
	edge [] (a3);
\node [noddee1] at (2.5,0) (s2) [label=below: \scriptsize $s_2$] {}
	edge [thick] (a)
	edge [thick] (b)
	edge [] (b1)
	edge [] (b2)
	edge [] (b3);
\node [noddee1] at (2.5,-1.25) (s3) [label=below: \scriptsize $s_1$]  {}
	edge [] (a1)
	edge [] (a2)
	edge [] (a3)
	edge [] (b1)
	edge [] (b2)
	edge [] (b3);
\node at (0,-3.6) {\scriptsize  $A\cong  K_{2p+1}$};
\node at (5,-3.6) {\scriptsize $B\cong K_{2q+1}$};	
\end{tikzpicture} }
\hspace*{1cm}
\subfigure[Type VII]{\label{fig: FC- ESE graph (bb)}
\begin{tikzpicture}[scale=.7]
\draw [rounded corners] (-.65,-3) rectangle (.65,3);
\draw[ultra thick] (-.65,1)--(.65,1);
\draw[ultra thick] (-.65,-1)--(.65,-1);
\node [noddee2] at (0,2.5) (a1)  {};
\node  at (0,2.1)  {\scriptsize $\vdots$};
\node [noddee2] at (0,1.5) (a11)  {};
\node [noddee2] at (0,.5) (a2)  {};
\node  at (0,.1)  {\scriptsize $\vdots$};
\node [noddee2] at (0,-.5) (a22)  {};
\node [noddee2] at (0,-1.5) (a3)  {};
\node  at (0,-1.9)  {\scriptsize $\vdots$};
\node [noddee2] at (0,-2.5) (a33)  {};
\node at (-1,2) {\scriptsize  $A_1$};
\node at (-1,0) {\scriptsize  $A_2$};
\node at (-1,-2) {\scriptsize  $A_3$};
\draw [rounded corners] (4.35,-3) rectangle (5.65,3);
\node [noddee2] at (5,1.5) (b)  {};	
\node [noddee2] at (5,0) (b1)  {};
\node  at (5,-1)  {\scriptsize $\vdots$};
\node [noddee2] at (5,-2) (b3)  {};
\node [noddee1] at (2.5,1.25) (s1) [label=above: \scriptsize $s_3$] {}
	edge [] (a1)
	edge [] (a11)
	edge [] (a2)
	edge [] (a22)
	edge [] (b)
	edge [] (b1)
	edge [] (b3);
\node [noddee1] at (2.5,0) (s2) [label=below: \scriptsize $s_2$] {}
	edge [] (a1)
	edge [] (a11)
	edge [] (a3)
	edge [] (a33)
	edge [] (b)
	edge [] (b1)
	edge [] (b3);
\node [noddee1] at (2.5,-1.25) (s3) [label=below: \scriptsize $s_1$]  {}
	edge [] (a2)
	edge [] (a22)
	edge [] (a3)
	edge [] (a33)
	edge [] (b)
	edge [] (b1)
	edge [] (b3);
\node at (0,-3.6) {\scriptsize  $A\cong  K_{2p+1}$};
\node at (5,-3.6) {\scriptsize $B\cong K_{2q+1}$};	
\end{tikzpicture} }
\subfigure[Type VIII]{\label{fig: FC- ESE graph (cc)}
\begin{tikzpicture}[scale=.7]
\draw [rounded corners] (-.65,-3) rectangle (.65,3);
\draw[ultra thick] (-.65,0)--(.65,0);
\node [noddee2] at (0,2.5) (a1)  {};
\node [noddee2] at (0,1.5) (a11)  {};
\node  at (0,1.1)  {\scriptsize $\vdots$};
\node [noddee2] at (0,.5) (a111)  {};
\node [noddee2] at (0,-.5) (a2)  {};
\node [noddee2] at (0,-1.5) (a22)  {};
\node  at (0,-1.9)  {\scriptsize $\vdots$};
\node [noddee2] at (0,-2.5) (a222)  {};
\node at (-1,1.5) {\scriptsize  $A_1$};
\node at (-1,-1.5) {\scriptsize  $A_2$};
\draw [rounded corners] (4.35,-3) rectangle (5.65,3);
\node [noddee2] at (5,1.5) (b)  {};	
\node [noddee2] at (5,0) (b1)  {};
\node  at (5,-1)  {\scriptsize $\vdots$};
\node [noddee2] at (5,-2) (b3)  {};
\node [noddee1] at (2.5,1.25) (s1) [label=above: \scriptsize $s_1$] {}
	edge [] (a1)
	edge [] (a11)
	edge [] (a111)
	edge [] (b)
	edge [] (b1)
	edge [] (b3);
\node [noddee1] at (2.5,0) (s2) [label=below: \scriptsize $s_2$] {}
	edge [] (a1)
	edge [] (a11)
	edge [] (a111)
	edge [] (b)
	edge [] (b1)
	edge [] (b3)
	edge [] (s1);
\node [noddee1] at (2.5,-1.25) (s3) [label=below: \scriptsize $s_3$]  {}
	edge [] (a2)
	edge [] (a22)
	edge [] (a222)
	edge [] (b)
	edge [] (b1)
	edge [] (b3);
\node at (0,-3.6) {\scriptsize $A\cong  K_{2p+1} $};
\node at (0,-4.2) {\scriptsize $ (\mbox{or } K_{2p})$};
\node at (5,-3.6) {\scriptsize $B\cong K_{2q+1}$};
\node at (5,-4.2) {\scriptsize $(\mbox{or } K_{2q})$};
\end{tikzpicture} }
\caption{Factor-critical $\ECE$-graphs with connectivity $3$.}
\label{fig:ECE, conn 3}
\end{figure}

\begin{nlem}\label{lem: FC ECE,  connectivity 3}
Let $G$ be a 3-connected graph with a 3-cut set $S$ such that $G - S$ has two components with at least 3 vertices. Then $G$ is factor-critical ECE if and only if it is of one of the three types depicted in Figure \ref{fig:ECE, conn 3}. 
\end{nlem}

\begin{proof}
Let $G$ be a factor-critical $\ECE$-graph, and let $S=\{s_1,s_2,s_3\}$ be a $3$-cut such that $G-S$ has two components with at least 3 vertices. By Theorem \ref{thm:kotbic conn 3}, $G - S$ has exactly two components $A,B$ and both are complete graphs. Clearly, $A$ and $B$ are either  both odd or both even since $G$ is odd. Recall also  that every vertex in $S$ has at least one neighbour in both $A$ and $B$ since $G$ is a graph with connectivity $3$. Observe that  for every vertex $a \in  A$, the graph $A-a$ has a perfect matching whenever $A$ is odd, we denote by $M_a$ such a perfect matching. By symmetry, the same holds for $B$ as well.
We will prove the claim under two main cases; $S$ is independent or not.\medskip

\textbf{\textit{Case 1:}} $S$ is an independent set. 

We first note that each one of $A$ and $B$ has odd order, since otherwise the graph $G-S$ consisting of two complete graphs has a perfect matching, contradicting to the equimatchability of $G$ by Lemma \ref{lem: defn equim}.  Thus we have  $A\cong K_{2p+1}$ and $B=K_{2q+1}$ for some $p,q\geq 1$.

We claim that every vertex in $A \cup B$ is adjacent to at least two vertices of $S$. Assume without loss of generality that there is a vertex $a \in A$ not adjacent to  $s_2,s_3 \in S$, then $\{a,s_2,s_3\}$ is an independent set. It follows that  for $b\in N_B(s_1)$,  the set $M_a \cup M_b \cup \{s_1b\}$ is a perfect matching in $G- \{a,s_2,s_3\}$, contradicting to the equimatchability of $G$ by Lemma \ref{lem: defn equim}. Hence the claim holds. \medskip

\textbf{\textit{$\quad$ Subcase 1.1:}} There is a vertex of $S$, without loss of generality $s_1$, such that $s_1a\notin E(G)$, $s_1b \notin E(G)$ for some $a \in A$ and $b \in B$.
 
Recall that every vertex in $A \cup B$ is adjacent to at least two vertices of $S$. 
Then both $a$ and $b$ are complete to $\{s_2,s_3\}$. If $s_2$ and $s_3$ have two disjoint neighbours in $A-a$, say  $a_1s_2,a_2s_3 \in E(G)$ for $a_1,a_2 \in A-a$, then there is a perfect matching $M_a$ of $A-a$ containing $a_1a_2$, and $(M_a - \{a_1a_2\}) \cup M_b \cup \{a_1s_2,a_2s_3\}$  is a perfect matching in $G - \{a,s_1,b\}$ where $\{a,s_1,b\}$ is an independent set, a contradiction to the equimatchability of $G$ by Lemma \ref{lem: defn equim}. This implies that one of $s_2,s_3$ is not adjacent to any  vertex in $A-a$ since every vertex in $A \cup B$ is adjacent to at least two vertices of $S$, say without loss of generality $s_2\notin N(A-a)$. That is, $s_2$ is adjacent to only $a$ in $A$, and $s_3$ is complete to $A-a$. Moreover, we infer that $s_1$ is complete to $A-a$.  On the other hand, we conclude by symmetry that one of $s_2,s_3$ is not adjacent to any  vertex in $B-b$. If $s_2$ is not adjacent to any  vertex in $B-b$, then both  $A-a$ and $B-b$ are complete to $\{s_1,s_3\}$, and the vertex $s_2$ is adjacent to only $a$ and $b$. In such a case, for $a_1,a_2\in A-a$ and $b_1\in B-b$, the set $(M_a- \{a_1a_2\}) \cup M_{b_1} \cup \{a_2s_1,as_3\}$  is a perfect matching in $G - \{a_1,s_2,b_1\}$, a contradiction by Lemma \ref{lem: defn equim}. Therefore, we may assume that  $s_2$ is complete to $B-b$, and $s_3$ is adjacent to only $b$ in $B$.  Hence $G$ is of Type VI in Figure \ref{fig: FC- ESE graph (aa)}. \medskip
  
\textbf{\textit{$\quad$ Subcase 1.2:}} Suppose that each vertex in $S$ is complete to $A$ or $B$.

In this case, all vertices of $S$ must be complete to the same part of $G -S$. Indeed, if $s_1$ and $s_2$ are complete to $B$, and $s_3$ is complete to $A$, then we observe that $N(\{s_3,b\})=V(G)$ for $b \in N_B(s_3)$, a contradiction with the criticality of the edge $s_3b$ by Corollary \ref{cor:no-dom-edge}. Therefore, all vertices of $S$ are complete to either $A$ or $B$, say without loss of generality $B$. Let now consider the edges between $A$ and $S$. If a vertex $a \in A$ is complete to $S$, then $N(\{a,s_1\})=V(G)$, contradicting the critically of the edge $as_1$ by Corollary \ref{cor:no-dom-edge}. Thus every vertex in $A$ is adjacent to exactly two vertices of $S$. Hence $G$ is of Type VII in Figure \ref{fig: FC- ESE graph (bb)}. \medskip

\textbf{\textit{Case 2:}} $S$ is not an independent set, let without loss of generality $s_1s_2 \in E(G)$.   \medskip 
  
\textbf{\textit{$\quad$ Subcase 2.1:}}  Both $A$ and $B$ are odd.
  
First we observe that the vertex $s_3$ is complete to $A$ or $B$. Indeed if there exist $a \in A$, $b \in B$ such that $as_3,bs_3 \notin E(G)$, then $G - \{a,s_3,b\}$ has a perfect matching $M_a \cup M_b \cup \{s_1s_2\} $, a contradiction with the equimatchability of $G$ by Lemma \ref{lem: defn equim}. Thus, assume without loss of generality that $s_3$ is complete to $B$. Remark that none of $s_1,s_2$ is complete to $A$ since otherwise letting without loss of generality $s_2$ to be complete to $A$, for $b \in N_B(s_2)$ we have $N(\{s_2,b\})=V(G)$, contradicting the critically of the edge $s_2b$ by Corollary \ref{cor:no-dom-edge}. Thus there exist  $a_1,a_2 \in A$ such that $a_1s_1,a_2s_2 \notin E(G)$.

Now we claim that $s_1$ and $s_2$ are complete to $B$. Assume to the contrary that $s_1$ is not complete to $B$. Then there exists $b_1 \in B$ such that $b_1s_1 \notin E(G)$. In this case, if $s_2$ is adjacent to a vertex $b_2 \in B-b_1$, then $\{a_1,s_1,b_1\}$ is an independent set, and there is a perfect matching of $B-b_1$ containing $b_2b_3$ for some $b_3 \in B-\{b_1,b_2\}$; then
$M_{a_1} \cup (M_{b_1}-b_2b_3) \cup \{s_2b_2,s_3b_3\}$ is a perfect matching in $G - \{a,s_1,b\}$, a contradiction to the equimatchability of $G$ by Lemma \ref{lem: defn equim}.  Then there is no such a vertex $b_2 \in B-b_1$, that is, $s_2$ is adjacent to only $b_1$  in $B$. It follows that for $b_3 \in N_{B}(s_1)$ and $b_2 \in B-\{b_1,b_3\}$, we have  $s_2a_1,s_2b_2 \notin E(G)$, and $M_{a_1} \cup (M_{b_1}-b_2b_3) \cup \{s_1b_3,s_3b_1\}$ is a perfect matching of $G - \{a_2,s_2,b_2\}$, again a contradiction by Lemma \ref{lem: defn equim}.  Consequently, $S$ is complete to $B$.

On the other hand, if there exists one of edges $s_1s_3$, $s_2s_3$, say $s_2s_3 \in E(G)$, then for $a \in N_A(s_2)$, we have $N(\{a,s_2\})=V(G)$, it is a contradiction to the edge-criticality of $G$ by Corollary \ref{cor:no-dom-edge}. Therefore, $s_3$ is not adjacent to $s_1$ and $s_2$. 
In the meanwhile, every vertex in $A$ is adjacent to at least one vertex of $S$. Indeed, if $a \in A$ has no neighbor in $ S$, then $\{a,s_1,s_3\}$ is an independent set. It follows that  for $b\in N_B(s_2)$,  the set $M_a \cup M_b \cup \{s_1b\}$ is a perfect matching in $G- \{a,s_2,s_3\}$, contradicting to the equimatchability of $G$ by Lemma \ref{lem: defn equim}. Hence the claim holds. 

Next we claim that every vertex in $A$ is complete to either $\{s_1,s_2\}$ or $\{s_3\}$. First, if a vertex $a \in A$ is adjacent to only one of $s_1,s_2$ in $S$, say $as_1,as_3 \notin E(G)$ and $as_2\in E(G)$, then for $b \in B$, the set $M_a \cup M_b \cup \{s_2b\}$ is a perfect matching in $G - \{a,s_1,s_3\}$, it is a contradiction with equimatchability of $G$  by Lemma \ref{lem: defn equim}.  Then, every vertex in $A$ is complete to  $\{s_3\}$ or two vertices in $\{s_1,s_2,s_3\}$.
Moreover, if the sets $\{s_1,s_2\}$ and $\{s_3\}$ have a common neighbour in $A$, say $a\in N(s_2)\cap N(s_3)$, then $N[\{a,s_2\}]=V(G)$, a contradiction by Lemma \ref{lem: critical edge iff there is a matching}. We then conclude that every vertex in $A$ is complete to either $\{s_1,s_2\}$ or $s_3$. 
Hence $G$ is of Type VIII in Figure \ref{fig: FC- ESE graph (cc)}.  \medskip
 
\textbf{\textit{$\quad$ Subcase 2.2:}}  Both $A$ and $B$ are even. 
  
Clearly, $|A|\geq 4$ and $|B|\geq 4$. First, if a vertex $s_i\in S$ is complete to neither $A$ nor $B$, say without loss of generality $s_1$, then there exist $a \in A$, $b \in B$ such that $as_1,bs_1 \notin E(G)$. Consider the independent set $\{ a,s_1,b \}$. If there are two vertices $a' \in A-a$ and $b' \in B-b$ such that either $a's_2,b's_3 \in E(G)$ or $a's_3,b's_2 \in E(G)$, assume without loss of generality $a's_2,b's_3 \in E(G)$, then since $A$ and $B$ are even cliques, they admit perfect matchings $M_1$ containing $aa'$ and $M_2$ containing $bb'$, respectively, such that $(M_1- \{aa'\})\cup (M_2- \{bb'\})\cup \{a's_2,b's_3\}$ is a perfect matching in $G - \{a,s_1,b\}$,  a contradiction with equimatchability of $G$ by Lemma \ref{lem: defn equim}. Therefore, we assume that there are no such vertices $a',b'$. This implies that either $s_2$ is adjacent to only $a$ in $A$ or $s_3$ is adjacent to only $b$ in $B$. 
The former implies that $S'=\{a,s_1,s_3\}$ is a 3-cut set, and $G - S'$ has two components $A-a$ and $B \cup \{s_1\} $ with at least $3$ vertices since  $|A|\geq 4$. Hence, boils down to the case where $A$ and $B$ are odd. Similarly, if a vertex of $S$ is adjacent to a unique vertex in $A$ or $B$, then it again turns the case that $A$ and $B$ are odd. Therefore we may assume that every vertex in $S$ is adjacent to at least two vertices in each one of $A$ and $B$. However, this contradicts again with the assumption that  there are no vertices $a',b'$ as described above. We therefore conclude that every vertex in $S$ is complete to $A$ or $B$. Moreover, all of them are complete to the same part of $A\cup B$, since otherwise if $s_2$ is complete to $A$ while $s_3$ is complete to $B$, then for $b \in N_B(s_2)$, we have $N(\{s_2,b\})=V(G)$, it is a contradiction to the edge-criticality of $G$ by Corollary \ref{cor:no-dom-edge}. Thus we can assume that $S$ is complete to $B$.  

Finally, we claim that every vertex of $A$ is complete to either $\{s_1,s_2\}$ or $\{s_3\}$. Similarly as above, if a vertex $a \in A$ is adjacent to only one of $s_1,s_2$ in $G$, say without loss of generality $as_1,as_3 \notin E(G)$ and $as_2\in E(G)$, then there exists $a' \in N_{A-a}(s_2)$ since every vertex in $S$ is adjacent to at least two vertices in each one of $A$ and $B$. It then follows that for a perfect matching $M_1$ of $A$ containing $aa'$ and a perfect matching $M_2$ of $B$, the set $(M_1-aa') \cup M_2 \cup \{s_2a'\}$ is a perfect matching in $G -\{a,s_1,s_3\}$, a contradiction with the equimatchability of $G$ by Lemma \ref{lem: defn equim}.  Moreover, if the sets $\{s_1,s_2\}$ and $\{s_3\}$ have a common neighbour in $A$, say $a\in N(s_2)\cap N(s_3)$, then $N(\{a,s_2\})=V(G)$, a contradiction to the edge-criticality of $G$ by Corollary \ref{cor:no-dom-edge}. Hence  every vertex in $A$ is complete to either $\{s_1,s_2\}$ or $\{s_3\}$. It follows that $G$ is of Type VIII in Figure \ref{fig: FC- ESE graph (cc)}.

Now, let us show the converse. For each class, one can check that the removal of no independent set of size 3 leaves a graph with a perfect matching, proving equimatchability by Lemma \ref{lem: defn equim}. Moreover, one can check that every edge is critical using Lemma \ref{lem: critical edge iff there is a matching}. We did these routine checks by computer for each type of configuration where the sets $A$ and $B$ are chosen as minimum representatives. Indeed, increasing their sizes would only create edges similar to the ones already considered. By a similar approach as in Lemma \ref{lem:All graphs belonging to F are ECE}, we obtain that all graphs having one of the three configurations in Figure \ref{fig:ECE, conn 3}  are $\ECE$.
\end{proof}

Next, we deal with the case where for a $k$-cut set $S$, one of the components in $G\sm S$ has exactly two vertices.

%
%
%

\begin{nthm}\cite{Kotbic}\label{thm:kotbic-k-conn-two-vertices}
Let $G$ be a $k$-connected equimatchable factor-critical graph with a $k$-cut $S$, where $k \geq 3$. Assume
that $G-S$ has a component $C$ with at least $k$ vertices and $G-(S \cup C)$ has a component with exactly two vertices. Then $G-S$ has exactly two components. Furthermore, if $S$ contains an edge, then $C$ is a complete graph. 
\end{nthm}

\begin{nlem}\label{lem:component-2-vertices}
Let $G$ be a $k$-connected factor-critical graph with at least $2k+3$ vertices for $k\geq 3$ and  a $k$-cut set $S$ such that $G - S$ has a component with exactly $2$ vertices. Then $G$ is ECE-graph if and only if it is of Type VIII in Figure \ref{fig: FC- ESE graph (cc)}.
\end{nlem}

\begin{proof}
Let $G$ be a factor-critical ECE-graph satisfying the assumptions of the statement. 
By Lemma \ref{lem:kotbic two-comp}, $G-S$ has precisely two components $A$ and $B$. By assumption,  one of $A$ and $B$ consists of two vertices. Let without loss of generality $A=\{a_1,a_2\}$. Then $B$ has at least $k+1$ vertices since $|G|\geq 2k+3$ and $k\geq 3$.  It follows from Theorem \ref{thm:kotbic-k-conn-two-vertices} that if $S$ contains an edge, then $B$ is a complete graph.

We first suppose that $S$ is an independent set. If $A$ is independent set too, then  each $a_i$ is complete to $S$ since $d(v)\geq k$ for each $v\in V(G)$.  Notice that there is no matching $M$ between $B$ and $S$ saturating all vertices of $S$,  since otherwise $G-V(M)$ is disconnected and contradicts Lemma \ref{lem: gen isolating matching}. In such a case, there exist a subset $R\subset S$ of size $r\leq k-1$ such that $|T=N(R)\cap B|\leq r-1$. However, then $T\cup (S-R)$ with $t\leq k-1$ is a $t$-cut set, a contradiction with $G$ being $k$-connected. Thus, we can assume that $A$ is not independent set.
Then each vertex $a_i$ is adjacent at least $k-1$ vertices in $S$. So, for each pair $s_i,s_j\in S$, the set $\{s_i,s_j\}$ can be matched into $A$.  Notice also that $B$ does not have a perfect matching since otherwise $G-S$ would have a perfect matching, contradicting to the equimatchability of $G$ by Lemma  \ref{lem: defn equim}. Consider a matching $M_i$ isolating $a_i$, it is clear from Lemma \ref{lem: gen isolating matching} that for each subset $R\subset B$ such that $R$ can be matched into $k-1$ vertices of $S$, the graph $B-R$ is a randomly matchable graph. We then deduce that $B$ is isomorphic to $K_{p,p+2}$ for $p\geq 1$ since  $B$ has no  perfect matching. It follows that each $s_i$ is adjacent to some vertices of the $(p+2)$-stable set of $B$.  Consider now a matching $M$ isolating $s_k$ such that $s_1a_1,s_2a_2\in M $, where we recall hat $S-s_k$ can be matched into $A$. Notice that $s_k$ cannot be adjacent to $p+1$ vertices of the $(p+2)$-stable set of $B$ since otherwise there is no matching containing $s_ka_1$ and  saturating $N(\{s_k,a_1\})$, a contradiction to the criticality of the edge $s_ka_1$ by Corollary \ref{cor:no-dom-edge}. So  $s_k$ is adjacent to at most $p$ vertices of the $(p+2)$-stable set of $B$. We then observe that $G-V(M)$ is not a randomly matchable graph, a contradiction by Lemma \ref{lem: gen isolating matching}. It follows that if $S$ is not an independent set, then $G$ is not $\ECE$.

Now we suppose that $S$ is not independent, and let without loss of generality $s_1s_2 \in E(G)$. By Theorem \ref{thm:kotbic-k-conn-two-vertices}, $B$ is a complete graph. Similarly as before, if $A$ is independent, then each $a_i$ is complete to $S$ since $d(v)\geq k$ for each $v\in V(G)$. It immediately follows that $G-(V(M\cup \{s_1s_2\})$ is a disconnected graph for a matching $M$ between $B$ and $S$ saturating $\{s_3,s_4,\ldots, s_k\}$, a contradiction. Thus $A$ is not independent, so $a_1a_2\in E(G)$.

Since $A$ and $B$ are even clique, each one has a perfect matching, say $M_1=\{a_1a_2\}$ and $M_2$, respectively. First, if a vertex $s_i\in S$ is complete to neither $A$ nor $B$, say $s=s_1$, then there exist $a \in A$, $b \in B$ such that $as_1,bs_1 \notin E(G)$, say $a=a_1$. Then  $a_1$ is complete to $S-s_1$. Since $G$ is ECE-graph, there exists a matching saturating $N(\{a_1,a_2\})$. It follows that there exists a matching $M$ between $B$ and $S$ and covering $\{s_3,s_4,\ldots,s_k\}$. Consider the independent set $\{ a,s_1,b \}$, and a perfect matching $M'$ in $B-V(M)$, then $ (M\cup M' \{a_2s_2\}$ is a perfect matching in $G - \{a_1,s_1,b\}$,  a contradiction with equimatchability of $G$ by Lemma \ref{lem: defn equim}.  Therefore, we assume that every vertex in $S$ is complete to $A$ or $B$. As we apply the same process in Lemma \ref{lem: FC ECE,  connectivity 3}, we obtain that $G$ is of Type VIII in Figure \ref{fig: FC- ESE graph (cc)}. \medskip
\end{proof}

It is clear that there is no $k$-connected factor-critical graph with at most $2k+1$ vertices for $k=3$ (see Figure \ref{fig:7-vertices}).  ~~\medskip

%
%
%
%
%
%

\subsection*{A2. 4-connected factor-critical ECE-graphs} ~~\medskip

We now deal with the edge-criticality of 4-connected equimatchable factor-critical graphs.

\begin{nlem} \cite{Kotbic} \label{lem:kotbic alpha=2}
Let $G$ be a $k$-connected equimatchable factor-critical graph with at least $2k+3$ vertices and a $k$-cut set $S$ such that $G - S$ has two components with at least 3 vertices, where $k \geq 4$. Then the independence number of $G$ is 2.
\end{nlem}

\begin{nthm}\cite{Kotbic}\label{kotbic-k-conn-single-vertex}
Let G be a $k$-connected equimatchable factor-critical graph with a $k$-cut $S$ such that $G-S$ has a component with a single vertex and a component $C$ with at least $k$ vertices, where $k \geq 2$. Then $G-S$ has exactly two components and there is a matching $M$ between $S$ and $C$ saturating all vertices of $S$. 
\end{nthm}

\begin{nprop}\cite{Kotbic}\label{prop:maximal -matching leaves at most 2 uncovered}
Let $G$ be a graph with independence number 2. If $G$ is odd, then $G$ is equimatchable. If $G$
is even, then either $G$ is randomly matchable, or $G$ is not equimatchable, has a perfect matching, and every maximal matching of $G$ leaves unsaturated at most two vertices.
\end{nprop}

\begin{nthm}\cite{Kotbic}\label{thm:kotbic-EFC-alpha-2}
Let $G$ be a $k$-connected odd graph with at least $2k + 3$ vertices and a $k$-cut set $S$ such that $G-S$ has two components with at least 3 vertices, where $k \geq 4$. Then $G$ has independence number at most 2 if and only if it is equimatchable and factor-critical.
\end{nthm}

%
%
%
%

A set $S$ of vertices is said to \emph{dominate} another set $T$ if every vertex in $T$ is adjacent to at least one vertex in $S$. An edge is called \emph{dominating edge} if the endpoints of the edge is a dominating set in $G$.

Note that a $k$-connected graph does not have a vertex of degree at most $k-1$, i.e., if a graph $G$ is $k$-connected then $\delta(G)\geq k$.

\begin{nlem}\label{lem:conn 4- three equiv}
Let $G$ be a $k$-connected odd graph with at least $2k+3$ vertices for $k\geq 4$ and $\delta(G)>k$. Then the following are equivalent.
\begin{itemize}
\item[$(i)$] $G$ is $\ECE$-graph.
\item[$(ii)$] $\alpha(G)=2$ and  $G$ has no dominating edge.
\item[$(iii)$] $\overline{G}$ is triangle-free and $diam(\overline{G})=2$.
\item[$(iv)$] $\overline{G}$ is a maximal triangle-free graph.
\end{itemize}
\end{nlem}

\begin{proof}
The equivalence of $(ii)$, $(iii)$ and $(iv)$ can be easily observed. So, we only need to show the equivalence of $(i)$ and $(ii)$.

First assume that $G$ is an $\ECE$-graph. Let $S$ be a $k$-cut set. By Lemma \ref{lem:kotbic two-comp}, $G-S$ has precisely two components $A$ and $B$. Also, both have at least two vertices since $\delta(G)> k$. If one of $A,B$ consists of exactly two vertices then we are done by Lemmas \ref{lem:component-2-vertices}. 
Thus assume that both $A$ and $B$ consist of at least $3$ vertices. It then follows from Lemma \ref{lem:kotbic alpha=2} that $\alpha(G)=2$. In addition,  $G$ has no dominating edge by Lemma \ref{lem: critical edge iff there is a matching} since $G$ is an $\ECE$-graph.

Suppose now that  $\alpha(G)=2$ and  $G$ has no dominating edge. Let $S$ be a $k$-cut set. Since $\alpha(G)=2$, the graph $G-S$ has precisely two components  $A$ and $B$, also both are complete graph.
Notice that $|A|\geq 2$ and $|B|\geq 2$ since $\delta(G)>k$. In addition, every vertex in  $S$ is complete to $A$ or $B$ because of $\alpha(G)=2$.
If one of $A,B$ consists of two vertices then we are done by Lemmas \ref{lem:component-2-vertices}. 
So we assume that $|A|\geq 3$ and $|B|\geq 3$. It then follows from Theorem \ref{thm:kotbic-EFC-alpha-2} that $G$ is factor-critical equimatchable graph. We shall show that $G$ is edge-critical graph, i.e., every edge is critical. Since there is no dominating edge, for every edge $xy$, there exists $z\in V(G)\setminus \{x,y\}$ such that $z$ is adjacent to neither $x$ nor $y$. It then suffices to show that $G-z$ has a perfect matching containing $xy$ for some $z\in V(G-N(\{x,y\}))$. We first note that every maximal matching leaves at most 2 unsaturated vertices in both $G$ and $G[S]$ by Proposition \ref{prop:maximal -matching leaves at most 2 uncovered} since $\alpha(G)=2$.

If $xy$ belongs to $A$ or $B$, then we are done by taking $z=a$ for any $a\in A$. Similarly if  $xy$ belongs to $G[S]$, then we have such a matching.  Finally, suppose that $xy$ links $S$ and $A$ (or $B$). Let $x\in A$ and $y\in S$. If $y$ is complete to $B$, then $z \in S$. In such a case, we have a perfect matching in $G-z$. Indeed, consider a maximal matching $M$ isolating $y,z$ in $G[S]$, the matching $M$ leaves at most 2 unsaturated vertices by Proposition \ref{prop:maximal -matching leaves at most 2 uncovered}, say $u,v$. Since $S$ is a $k$-cut set, the set $\{u,v\}$ has at least two neighbours in both $A$ and $B$. Thus, we can extend $M\cup \{xy\}$ to a perfect matching in $G-z$ as we claimed. 
On the other hand, if $y$ is not complete to $B$, i.e., $y$ is not adjacent to $z=a\in A$, then we have a perfect matching containing $xy$ in $G-z$ by Proposition \ref{prop:maximal -matching leaves at most 2 uncovered}.
\end{proof}

%
%
%

\end{document}